\documentclass[10pt]{amsart}
      \usepackage[mathscr]{eucal}
      \usepackage{amsmath,amsfonts}
      \parskip\smallskipamount
          \newtheorem{theorem}{Theorem}[section]
      
      \newtheorem{proposition}[theorem]{Proposition}
      \newtheorem{corollary}[theorem]{Corollary}
      \newtheorem{lemma}[theorem]{Lemma}

      \makeatletter
      \@addtoreset{equation}{section}
      \makeatother

      \newcommand{\BB}{{\mathbb B}}
      \newcommand{\CC}{{\mathbb C}}
      \newcommand{\NN}{{\mathbb N}}

      \newcommand{\DD}{{\mathbb D}}
      \newcommand{\RR}{{\mathbb R}}
      \newcommand{\FF}{{\mathbb F}}
      \newcommand{\TT}{{\mathbb T}}
      
      \newcommand{\HH}{{\mathbb H}}

      \newcommand{\cA}{{\mathcal A}}

      \newcommand{\cD}{{\mathcal D}}
      \newcommand{\cE}{{\mathcal E}}
      
      \newcommand{\cG}{{\mathcal G}}
      \newcommand{\cH}{{\mathcal H}}
      
      \newcommand{\cK}{{\mathcal K}}
      \newcommand{\cL}{{\mathcal L}}

      \newcommand{\cR}{{\mathcal R}}

      \newdimen\expt
      \def\boxit#1{\setbox0\hbox{$\displaystyle{#1}$}
            \hbox{\lower.4\expt
       \hbox{\lower3\expt\hbox{\lower\dp0
            \hbox{\vbox{\hrule height.4\expt
       \hbox{\vrule width.4\expt\hskip3\expt
            \vbox{\vskip3\expt\box0\vskip2\expt}%
       \hskip3\expt\vrule width.4\expt}\hrule height.4\expt}}}}}}
      
\begin{document}
       \pagestyle{myheadings}
      \markboth{ Gelu Popescu}{ Hyperbolic  geometry
       on the unit ball of $B(\cH)^n$ and dilation theory }


      \title [ Hyperbolic  geometry
       on the unit ball of $B(\cH)^n$ and dilation theory]
{      Hyperbolic  geometry
       on the unit ball of $B(\cH)^n$ and dilation theory
       }
        \author{Gelu Popescu}

\date{September 15, 2008}
      \thanks{Research supported in part by an NSF grant}

      \address{Department of Mathematics, The University of Texas
      at San Antonio \\ San Antonio, TX 78249, USA}
      \email{\tt gelu.popescu@utsa.edu}

\begin{abstract}

 In this paper we continue our investigation
concerning the hyperbolic geometry  on the noncommutative ball
$$
[B(\cH)^n]_1^-:=\left\{(X_1,\ldots, X_n)\in B(\cH)^n:\ \|X_1
X_1^*+\cdots +X_nX_n^* \|^{1/2} \leq 1\right\},
$$
where $B(\cH)$ is  the algebra
  of all bounded linear operators on a Hilbert space $\cH$,
and  its implications to noncommutative function theory. The central
object is an intertwining operator $L_{B,A}$ of the minimal
isometric dilations of $A, B\in [B(\cH)^n]_1^-$, which establishes a
strong connection between  noncommutative hyperbolic geometry  on
$[B(\cH)^n]_1^-$ and multivariable dilation theory. The goal of this
paper is to study  the operator $L_{B,A}$ and its connections  to
 the hyperbolic metric  $\delta$ on the Harnack parts  $\Delta$ of $[B(\cH)^n]_1^-$.
 In particular, we show that
 \begin{equation*}
\begin{split}
\delta(A,B)&=
  \ln \max \left\{ \left\| L_{A,B} \right\|,
  \left\| L_{A,B}^{-1}\right\|\right\}
\end{split}
\end{equation*}
for any  $A,B\in \Delta$, and express $\|L_{B,A}\|$ in terms of the
reconstruction  operators $R_A$ and $R_B$.  We study the geometric
structure of the operator $L_{B,A}$ and obtain new characterizations
for the Harnack domination (resp.~equivalence) in $[B(\cH)^n]_1^-$.

Finally, given a contractive free holomorphic function
$F:=(F_1,\ldots, F_m)$ with coefficients in $B(\cE)$
 and  $z,\xi\in \BB_n$, the open unit ball of $\CC^n$,  we prove that  $F(z)$ is Harnack equivalent to $F(\xi)$ and
$$
\|L_{F(z), F(\xi)}\|\leq \|L_{z,\xi}\|=
\left(\frac{1+\|\varphi_z(\xi)\|_2}{1-\|\varphi_z(\xi)\|_2}\right)^{1/2},$$
where $\varphi_z$ is the involutive automorphism of $\BB_n$ which
takes $0$ into $z$. This result implies a Schwartz-Pick lemma for
 operator-valued   multipliers  of the  Drury-Arveson space, with respect to the
 hyperbolic metric.
\end{abstract}

      \maketitle

\bigskip

\bigskip

\section*{Introduction}

In \cite{Po-hyperbolic}, we introduced a hyperbolic   metric
$\delta$ on the noncommutative ball
$$
[B(\cH)^n]_1:=\left\{(X_1,\ldots, X_n)\in B(\cH)^n:\ \|X_1
X_1^*+\cdots +X_nX_n^* \|^{1/2} <1\right\}.
$$
 This metric
    is a noncommutative extension  of the Poincar\'
  e-Bergman (\cite{Be}) metric  $\beta_n$  on the open unit ball
$\BB_n:=\{ z\in \CC^n:\ \|z\|_2<1\},$
  which  is defined by
$$
\beta_n(z,w):=\frac{1}{2}\ln
\frac{1+\|\varphi_z(w)\|_2}{1-\|\varphi_z(w)\|_2},\qquad z,w\in
\BB_n,
$$
where $\varphi_z$ is the involutive automorphism of $\BB_n$ that
interchanges $0$ and $z$. We  proved that $\delta$ is  invariant
under the action of  the group  $Aut([B(\cH)^n]_1)$ of all  free
holomorphic automorphisms of $[B(\cH)^n]_1$, and  showed that the
$\delta$-topology and the usual operator norm topology coincide on
$[B(\cH)^n]_1$.  Moreover, we proved that   $[B(\cH)^n]_1$ is a
complete metric space with respect to the hyperbolic metric and
obtained an explicit formula for $\delta$ in terms of the
reconstruction operator.  A Schwarz-Pick lemma for  bounded free
holomorphic functions on $[B(\cH)^n]_1$, with respect to the
hyperbolic metric, was also obtained. The  results from
\cite{Po-hyperbolic} make connections between noncommutative
function theory (see \cite{Po-poisson}, \cite{Po-holomorphic},
\cite{Po-automorphism}, \cite{Po-pluriharmonic}) and   classical
results in hyperbolic complex analysis and geometry (see \cite{Ko1},
\cite{Ko2}, \cite{Kr},  \cite{Ru}, \cite{Zhu}).

In this paper we continue to study   the noncommutative hyperbolic
geometry on the unit ball of $B(\cH)^n$, its connections with
multivariable dilation theory, and its implications
 to noncommutative function theory.

To present our results, we  recall (\cite{Po-hyperbolic})  a few
definitions. Let $A:=(A_1,\ldots, A_n)$   and  $B:=(B_1,\ldots,
B_n)$ be in $[B(\cH)^n]_1^-$.  We say that $A$ is Harnack dominated
by $B$ if there exists a constant $c>0$ such that
$$\text{\rm Re}\,p(A_1,\ldots, A_n)\leq c^2 \text{\rm
Re}\,p(B_1,\ldots, B_n) $$
 for any noncommutative polynomial with
matrix-valued coefficients $p\in \CC[X_1,\ldots, X_n]\otimes M_{m}$,
$m\in \NN$, such that $\text{\rm Re}\,p\geq 0$, i.e., $\text{\rm
Re}\,p(X_1,\ldots,X_n)\geq 0$ for any $(X_1,\ldots, X_n)\in
[B(\cK)^n]_1$, where $\cK$ is an infinite dimensional Hilbert space.
In this case, we denote  $A\overset{H}{{\underset{c}\prec}}\, B$.

In Section 1, we obtain a characterization of the Harnack domination
$A\overset{H}{{\underset{c}\prec}}\, B$ in terms of an intertwining
operator $L_{B,A}$ of the minimal isometric dilations of the row
contractions $A$ and $B$. More precisely, we prove that if
 $V:=(V_1,\ldots, V_n)$ on
$\cK_A\supseteq \cH$ and $W:=(W_1,\ldots, W_n)$ on $\cK_B\supseteq
\cH$ are  the minimal isometric dilations of $A$ and $B$,
respectively, and
 $c\geq1$, then   $A\overset{H}{{\underset{c}\prec}}\, B$ if and
 only if
there is an operator $L_{B,A}\in B(\cK_B, \cK_A)$ with
$\|L_{B,A}\|\leq c$ such that $L_{B,A}|_\cH=I_\cH$ and
\begin{equation*}
\label{L}
 L_{B,A} W_i=V_iL_{B,A},\qquad i=1,\ldots,n.
\end{equation*}
The central object of this paper is the intertwining operator
$L_{B,A}$ which establishes a bridge between
  noncommutative hyperbolic
geometry on the unit ball of $B(\cH)^n$ and    multivariable
dilation theory.

In Section 2, we   obtain some  results concerning the  Harnack
domination on $[B(\cH)^n]_1^-$ and
 the reconstruction operator
 $$
R_X:=X_1^*\otimes R_1+\cdots +X_n^*\otimes R_n, \quad
X:=(X_1,\ldots, X_n)\in [B(\cH)^n]_1,
$$
associated with the right creation operators $R_1,\ldots, R_n$ on
the full Fock space with $n$ generators. Using basic facts about
noncommutative Poisson transforms (\cite{Po-poisson}) on
Cuntz-Toeplitz algebras  (\cite{Cu}), we prove that if  $A$ and $B$
are in $[B(\cH)^n]_1^-$ and $c>0$, then
$A\overset{H}{{\underset{c}\prec}}\, B$ if and only if
$R_A\overset{H}{{\underset{c}\prec}}\, R_B$. This result is needed
in the next sections.

We recall  (\cite{Po-hyperbolic}) that  $A,B\in [B(\cH)^n]_1^-$ are
called Harnack equivalent (and denote
$A\overset{H}{{\underset{c}\sim}}\, B$)  if and only if there exists
$c\geq 1$ such that
\begin{equation*}
 \frac{1}{c^2}\text{\rm Re}\,p(B_1,\ldots, B_n)\leq
\text{\rm Re}\,p(A_1,\ldots, A_n)\leq c^2 \text{\rm
Re}\,p(B_1,\ldots, B_n)
\end{equation*}
for any noncommutative polynomial with matrix-valued coefficients
$p\in \CC[X_1,\ldots, X_n]\otimes M_{m}$, $m\in \NN$, such that
$\text{\rm Re}\,p\geq 0$.
 The equivalence classes with respect to $\overset{H}\sim$ are
called Harnarck parts of $[B(\cH)^n]_1^-$ and are   noncommutative
analogues of the Gleason parts of the Gelfand spectrum of a function
algebra (see \cite{G}). In Section 3, using  noncommutative dilation
theory for row contractions (\cite{Po-isometric},
\cite{Po-charact}), we obtain several results concerning the Harnack
 parts of $[B(\cH)^n]_1^-$.

We introduced in \cite{Po-hyperbolic} a hyperbolic metric $\delta$
on any Harnack part $\Delta$ of $[B(\cH)^n]_1^-$  by setting
\begin{equation*}
 \delta(A,B):=\ln \inf\left\{ c \geq  1: \
A\,\overset{H}{{\underset{c}\sim}}\, B   \right\}, \quad  A,B\in
\Delta.
\end{equation*}
In Section 4, we express the hyperbolic distance in terms of the
operator $L_{B,A}$ by  proving  that
\begin{equation*}
\begin{split}
\delta(A,B)&=
  \ln \max \left\{ \left\| L_{A,B} \right\|,
  \left\| L_{A,B}^{-1}\right\|\right\}
\end{split}
\end{equation*}
for any  $A,B\in \Delta$.
  Using a characterization of the Harnack equivalence  on
$[B(\cH)^n]_1^-$ in terms of free pluriharmonic kernels (see
\cite{Po-hyperbolic}), we  obtain an explicit formula for the  norm
of the intertwining operator $L_{B,A}$ in terms of the
reconstruction operators $R_A$ and $R_B$.
 More precisely,  we prove
that if $A,B\in [B(\cH)^n]_1$, then
$$
\|L_{B,A}\|=\|C_AC_B^{-1}\|,
$$
where $C_X:=( \Delta_X\otimes I)(I-R_X)^{-1}$. Similar result holds
if $A,B\in [B(\cH)^n]_1^-$  and $A\overset{H}{\prec}\, B$.

Section 5 deals with the geometric structure of the operator
$L_{B,A}$.  We show that  if $A,B\in [B(\cH)^n]_1^-$, then
$A\overset{H}{{\prec}}\, B$ if and only if there exit two bounded
linear operators
 $\Omega$ and $\Theta$ satisfying certain  analytic properties such that
$$
L_{B,A}^*=\left[\begin{matrix} I_\cH& 0\\
\Omega &  \Theta
\end{matrix}\right].
$$
As consequences, we obtain new characterizations  for the Harnack
domination (resp.~equivalence) in $[B(\cH)^n]_1^-$.

In Section 6, we obtain  a Schwartz-Pick  lemma for contractive free
holomorphic  functions on $[B(\cH)^n]_1$ with respect to the
intertwining operator $L_{B,A}$. More precisely,  given a
contractive free holomorphic function  $F:=(F_1,\ldots, F_m)$ with
coefficients in $B(\cE)$
 and  $z,\xi\in \BB_n$, we prove that  $F(z)\overset{H}{\sim}\, F(\xi)$ and
$$
\|L_{F(z), F(\xi)}\|\leq \|L_{z,\xi}\|. $$
As a consequence, we
deduce that
$$
\delta(F(z), F(\xi))\leq \delta(z,\xi),
$$
 where $\delta$ is the hyperbolic metric. We mention that the map
 $\BB_n\ni z\mapsto F(x)\in B(\cE)^m$ is a contractive operator-valued
 multiplier of  the  Drury-Arveson space (see \cite{Dr}, \cite{Arv}).

Finally, we want to acknowledge that we were strongly influenced in
writing this paper  and \cite{Po-hyperbolic} by the work of  Foia\c
s concerning Harnack parts of contractions  (\cite{Fo}) and  by the
work of Suciu   on the hyperbolic distance between two contractions
(\cite{Su1}, \cite{Su2}). The present  paper is dedicated to Ciprian
Foia\c s on his $75^{\text{\rm th}}$ anniversary.

\bigskip

\section{Harnack type domination on $[B(\cH)^n]_1^-$ and dilations}

In this section, we obtain  characterizations for  the Harnack
domination $A\overset{H}{{\underset{c}\prec}}\, B$ in terms of
multi-Toeplitz kernels on free semigroups, and in terms of an
intertwining operator $L_{B,A}$ of the minimal isometric dilations
of the row contractions $A$ and $B$.

Let $H_n$ be an $n$-dimensional complex  Hilbert space with
orthonormal
      basis
      $e_1$, $e_2$, $\dots,e_n$, where $n=1,2,\dots$, or $n=\infty$.
       We consider the full Fock space  of $H_n$ defined by
      $$F^2(H_n):=\CC1\oplus \bigoplus_{k\geq 1} H_n^{\otimes k},$$
      where  $H_n^{\otimes k}$ is the (Hilbert)
      tensor product of $k$ copies of $H_n$.
      Define the left  (resp.~right) creation
      operators  $S_i$ (resp.~$R_i$), $i=1,\ldots,n$, acting on $F^2(H_n)$  by
      setting
      $$
       S_i\varphi:=e_i\otimes\varphi, \quad  \varphi\in F^2(H_n),
      $$
       (resp.~$
       R_i\varphi:=\varphi\otimes e_i, \quad  \varphi\in F^2(H_n).
      $)
The noncommutative disc algebra $\cA_n$ (resp.~$\cR_n$) is the norm
closed algebra generated by the left (resp.~right) creation
operators and the identity. The   noncommutative analytic Toeplitz
algebra $F_n^\infty$ (resp.~$R_n^\infty$)
 is the  weakly
closed version of $\cA_n$ (resp.~$\cR_n$). These algebras were
introduced (\cite{Po-von} and \cite{Po-disc}) in connection with a
noncommutative von Neumann  type inequality (see \cite{vN} for the
classical case $n=1$). For basic results concerning completely
bounded maps
 and operator spaces we refer to \cite{Pa-book}, \cite{Pi}, and \cite{ER}.

 Let $\FF_n^+$ be the unital free semigroup on $n$ generators
$g_1,\ldots, g_n$ and the identity $g_0$.  The length of $\alpha\in
\FF_n^+$ is defined by $|\alpha|:=0$ if $\alpha=g_0$ and
$|\alpha|:=k$ if
 $\alpha=g_{i_1}\cdots g_{i_k}$, where $i_1,\ldots, i_k\in \{1,\ldots, n\}$.
If $(X_1,\ldots, X_n)\in B(\cH)^n$, where $B(\cH)$ is the algebra of
all bounded linear operators on the Hilbert space $\cH$,    we set
$X_\alpha:= X_{i_1}\cdots X_{i_k}$  and $X_{g_0}:=I_\cH$. We denote
$e_\alpha:= e_{i_1}\otimes\cdots \otimes e_{i_k}$  if
$\alpha=g_{i_1}\cdots g_{i_k}\in \FF_n^+$ and $e_{g_0}:=1$.  Note
that $\{e_\alpha\}_{\alpha\in \FF_n^+}$ is an orthonormal basis for
the full Fock space  $F^2(H_n)$.

Free pluriharmonic functions arise in the study of free holomorphic
functions on the noncommutative  open  unit ball $[B(\cH)^n]_1$. We
recall \cite{Po-holomorphic}  that a map $F:[B(\cH)^n]_{1}\to
B(\cH)\otimes_{min} B( \cE)$ is
   free
holomorphic function on  $[B(\cH)^n]_{1}$  with coefficients in
$B(\cE)$ if there exist $A_{(\alpha)}\in B(\cE)$, $\alpha\in
\FF_n^+$, such that  $\limsup_{k\to \infty} \left\|
\sum_{|\alpha|=k} A_{(\alpha)}^* A_{\alpha)}\right\|^{1/2k}\leq 1$
and $ F(X_1,\ldots, X_n)=\sum\limits_{k=0}^\infty
\sum\limits_{|\alpha|=k}X_\alpha \otimes A_{(\alpha)}, $ where the
series converges in the operator  norm topology  for any
$(X_1,\ldots, X_n)\in [B(\cH)^n]_{1}$. We say that
$h:[B(\cH)^n]_1\to B(\cH)\otimes_{min} B( \cE)$ is a self-adjoint
free pluriharmonic function on $[B(\cH)^n]_1$ if $h=\text{\rm Re}\,
f$ for some free holomorphic function $f$. An arbitrary free
pluriharmonic function is a linear combination of self-adjoint free
pluriharmonic functions.
 A free pluriharmonic functions on the
noncommutative ball
 $[B(\cH)^n]_1$  has the form
  \begin{equation*}  G(X_1,\ldots,
X_n)=\sum_{k=1}^\infty \sum_{|\alpha|=k}X_\alpha^*\otimes
B_{(\alpha)}
 +I\otimes A_{(0)}\otimes  + \sum_{k=1}^\infty
 \sum_{|\alpha|=k}X_\alpha\otimes A_{(\alpha)},
\end{equation*}
where the series are convergent in the  operator norm topology
  for any $ (X_1,\ldots, X_n)\in [B(\cH)^n]_1$.
   We
also recall that $G$ is called positive if $G(X_1,\ldots,X_n)\geq 0$
for any $(X_1,\ldots, X_n)\in [B(\cK)^n]_1$, where $\cK$ is an
infinite dimensional Hilbert space.

   Let
$A:=(A_1,\ldots, A_n)$ and $B:=(B_1,\ldots, B_n)$ be in
$[B(\cH)^n]_1^-$. As in  \cite{Po-hyperbolic}, we say that $A$ is
Harnack dominated by $B$, and denote $A\overset{H}{\prec}\, B$, if
there exists a constant $c>0$ such that
$$\text{\rm Re}\,p(A_1,\ldots, A_n)\leq c^2 \text{\rm
Re}\,p(B_1,\ldots, B_n) $$
 for any noncommutative polynomial with
matrix-valued coefficients $p\in \CC[X_1,\ldots, X_n]\otimes M_{m}$,
$m\in \NN$, such that $\text{\rm Re}\,p\geq 0$.
 When we
want to emphasize the constant $c$, we write
$A\overset{H}{{\underset{c}\prec}}\, B$.
In \cite{Po-hyperbolic}, we proved that  if $A,B\in [B(\cH)^n]_1^-$
and $c>0$, then
\begin{equation}
\label{equi}
 A\overset{H}{{\underset{c}\prec}}\, B\quad \text{ if
and only if } \quad P( rA, R)\leq c^2 P(rB, R)\quad \text{ for any }
\quad r\in [0,1),
\end{equation}
where $P(X, R)$ is the free pluriharmonic  Poisson kernel associated
with $X:=(X_1,\ldots, X_n) \in [B(\cH)^n]_1$ given by
$$
P(X,R):=\sum_{k=1}^\infty\sum_{|\alpha|=k}
 X_\alpha^*\otimes R_{\widetilde\alpha}
+I+\sum_{k=1}^\infty\sum_{|\alpha|=k} X_\alpha\otimes
R_{\widetilde\alpha}^*
$$
and the series are convergent in the operator norm topology.

In what follows we obtain a characterization of Harnack domination
in terms of multi-Toeplitz kernels on free semigroups.

A few more notations and definitions are necessary. If $\omega,
\gamma\in \FF_n^+$, we say that  $\omega
>_{l}\gamma$ if there is $\sigma\in
\FF_n^+\backslash\{g_0\}$ such that $\omega= \gamma \sigma$ and set
$\omega\backslash_l \gamma:=\sigma$.
 We denote by
$\tilde\alpha$  the reverse of $\alpha\in \FF_n^+$, i.e.,
  $\tilde \alpha= g_{i_k}\cdots g_{i_k}$ if
   $\alpha=g_{i_1}\cdots g_{i_k}\in\FF_n^+$.
An operator-valued  positive semidefinite kernel on the free
semigroup $\FF_n^+$ is a map
$ K:\FF_n^+\times\FF_n^+\to B(\cH) $ with the property that for each
$k\in\NN$, for each choice of vectors
 $h_1,\dots,h_k$ in $\cH$, and
$\sigma_1,\dots,\sigma_k$ in $\FF_n^+$ the inequality
$$
\sum\limits_{i,j=1}^k\langle K(\sigma_i,\sigma_j)h_j,h_i\rangle\ge 0
$$
holds. Such a kernel is called multi-Toeplitz if it has the
following properties:
$K(\alpha,\alpha)=I_\cH$  for  any $\alpha\in \FF_n^+$
and
$$
K(\sigma,\omega)=
\begin{cases}K(g_0,\omega \backslash_l \sigma)\ &\text{ if } \omega>_l\sigma  \\
 K(\sigma\backslash_l\omega,g_0) \ &\text{ if }\sigma>_l\omega  \\
0 &\text { otherwise. } \end{cases}
$$
If $X:=(X_1,\ldots, X_n)\in [B(\cH)^n]_1^-$, we define the
multi-Toeplitz kernel
   $K_{X}:\FF_n^+\times \FF_n^+\to B(\cH)$
   by
   $$
   K_{X}(\alpha, \beta):=
   \begin{cases}
   X_{\beta\backslash_l \alpha} &\text{ if } \beta>_l\alpha\\
   I  &\text{ if } \alpha=\beta\\
    X_{\alpha\backslash_l \beta}^*  &\text{ if } \alpha>_l\beta\\
    0\quad &\text{ otherwise}.
   \end{cases}
   $$

The first result is a characterization for  Harnack domination in
$[B(\cH)^n]_1^-$ in terms of multi-Toeplitz kernels on free
semigroups.

\begin{theorem}
\label{equivalent} Let $A:=(A_1,\ldots, A_n)$ and $B:=(B_1,\ldots,
B_n)$ be in $[B(\cH)^n]_1^-$ and let $c>0$. Then the following
statements are equivalent:
\begin{enumerate}
\item[(i)] $A\overset{H}{{\underset{c}\prec}}\, B$;
\item[(ii)]
$K_A\leq c^2 K_B$, where $K_X$ is the multi-Toeplitz kernel
associated with $X\in [B(\cH)^n]_1^-$;
\end{enumerate}
\end{theorem}

\begin{proof} We
recall that  $e_\alpha:= e_{i_1}\otimes\cdots \otimes e_{i_k}$  if
$\alpha=g_{i_1}\cdots g_{i_k}\in \FF_n^+$ and $e_{g_0}:=1$, and
that $\{e_\alpha\}_{\alpha\in \FF_n^+}$ is an orthonormal basis for
the full Fock space  $F^2(H_n)$. We prove  now that $(i)\implies
(ii)$.
 First,  we show   that
  \begin{equation}\label{Ar}
  \left< P(X, rR)\left( \sum_{|\beta|\leq q}  h_\beta\otimes e_\beta \right),
   \sum_{|\gamma|\leq q}  h_\gamma\otimes e_\gamma\right>
   = \sum_{|\beta|, |\gamma|\leq q}\left< K_{X,r}(\gamma, \beta)
    h_\beta, h_\gamma\right>,
  \end{equation}
  where
  the multi-Toeplitz kernel  $K_{X,r}:\FF_n^+\times \FF_n^+\to
  B(\cH)$, $r\in (0,1)$,
  is defined  by
   $$
   K_{X,r}(\alpha, \beta):=
   \begin{cases}
  r^{|\beta\backslash_l \alpha|} X_{\beta\backslash_l \alpha}
   &\text{ if } \beta>_l\alpha\\
   I  &\text{ if } \alpha=\beta\\
    r^{|\alpha\backslash_l \beta|}(X_{\alpha\backslash_l \beta})^*
     &\text{ if } \alpha>_l\beta\\
    0\quad &\text{ otherwise}.
   \end{cases}
   $$
   Indeed,  if $\{h_\beta\}_{|\beta|\leq q}\subset \cH$, then
    \begin{equation*}
    \begin{split}
    \left< \left(\sum_{k=0}^\infty \sum_{|\alpha|=k} X_{\alpha}^*\otimes r^k
     R_{\tilde\alpha}\right)
     \right.& \left.\left(\sum_{|\beta|\leq q}  h_\beta\otimes e_\beta \right),
   \sum_{|\gamma|\leq q} h_\gamma\otimes e_\gamma\right>\\
   &=
     \sum_{k=0}^\infty \sum_{|\alpha|=k}\left<\sum_{|\beta|\leq q}
      X_{\alpha}^*h_\beta\otimes r^k
     R_{\tilde\alpha} e_\beta,
    \sum_{|\gamma|\leq q}  h_\gamma\otimes e_\gamma\right> \\
    &=\sum_{\alpha\in \FF_n^+}\sum_{|\beta|, |\gamma|\leq q}
    r^{|\alpha|} \left< e_{\beta  {\alpha}}, e_\gamma\right>
    \left<X_{\alpha}^* h_\beta, h_\gamma\right>\\
    &=
    \sum_{ \gamma\geq\beta; ~|\beta|, |\gamma|\leq q}
    r^{|\gamma\backslash_l \beta|}
     \left<X_{\gamma\backslash_l \beta}^* h_\beta, h_\gamma\right>\\
     &=
     \sum_{\gamma\geq\beta; ~|\beta|, |\gamma|\leq q}\left<K_{X,r}
      (\gamma, \beta)h_\beta, h_\gamma\right>.
    \end{split}
    \end{equation*}
   Hence, taking into account that $K_{X,r}
      (\gamma, \beta)=K_{X,r}^*
      ( \beta, \gamma)$,  we deduce relation \eqref{Ar}.
   Consequently, the condition $P( rA, R)\leq c^2 P( rB,R)$ for any $r\in
   [0,1)$ implies
$$[K_{A,r}(\alpha,\beta)]_{|\alpha|, |\beta|\leq q}
  \leq c^2[K_{B,r}(\alpha,\beta)]_{|\alpha|, |\beta|\leq q} $$
  for any $0<r<1$
  and $q=0,1,\ldots$.
  Taking $r\to 1$, we   obtain item (ii).

  Assume now that (ii) holds. Since $c^2K_B-K_A$ is a positive
  semidefinite multi-Toeplitz operator, due to \cite{Po-posi} (see also
  the proof of Theorem 5.2 from \cite{Po-pluriharmonic}), we find a
  completely positive linear map $\mu:C^*(S_1,\ldots, S_n)\to
  B(\cE)$ such that
  $$
  \mu(S_\alpha)=c^2K_B(g_0, \alpha)-K_A(g_0, \alpha)=c^2B_\alpha-A_\alpha
  $$
  for any $\alpha \in \FF_n^+$. Since $P(rS, R)\geq 0 $ for $r\in
  [0,1)$, we deduce that
  \begin{equation*} \begin{split}
(\mu\otimes \text{\rm
 id})[P(rS, R)]&=\sum_{k=1}^\infty\sum_{|\alpha=k}
 r^{|\alpha|}[c^2 B_\alpha^*-A_\alpha^*]\otimes R_{\widetilde\alpha}
+(c^2B_0-A_0)\otimes I+\sum_{k=1}^\infty\sum_{|\alpha=k}
 r^{|\alpha|}[c^2 B_\alpha-A_\alpha]\otimes R_{\widetilde\alpha}^*\\
&=c^2P(rB, R)-P(rA, R)\geq 0,
  \end{split}
  \end{equation*}
which, due to \eqref{equi}, implies   (i).    This completes the
proof.
\end{proof}

An $n$-tuple  $T:=(T_1,\dots, T_n)$ of bounded linear  operators
acting on a common Hilbert space $\cH$
   is called
contractive (or row contraction) if
$$
T_1T_1^*+\cdots +T_nT_n^*\leq I_\cH.
$$
For simplicity, throughout this paper, $[T_1,\ldots, T_n]$ denotes
either the $n$-tuple $(T_1,\ldots, T_n)$ or  the operator row matrix
$[T_1\ \cdots \ T_n]$.
The defect operators  associated with a row contraction
$T:=[T_1,\ldots, T_n]$ are
\begin{equation}
\label{DelDel}
 \Delta_T:=\left( I_\cH-\sum_{i=1}^n
T_iT_i^*\right)^{1/2}\in B(\cH) \quad \text{ and }\quad
\Delta_{T^*}:=([\delta_{ij}I_\cH-T_i^*T_j]_{n\times n})^{1/2}\in
B(\cH^{(n)}),
\end{equation}
while the defect spaces are $\cD_T:=\overline{\Delta_T\cH}$ and
$\cD_{T^*}:=\overline{\Delta_{T^*}\cH^{(n)}}$, where $\cH^{(n)}$
denotes the direct sum of $n$ copies of $\cH$.
We say that an $n$-tuple  $V:=[V_1,\dots, V_n]$ of isometries on a
Hilbert space $\cK_T\supseteq \cH$ is a  minimal isometric dilation
of $T$ if the following properties are satisfied:
\begin{enumerate}
\item[(i)]  $V_1V_1^*+\cdots +V_nV_n^*\le I_{\cK_T};$
\item[(ii)]  $V_i^*|_\cH=T_i^*, \ i=1,\dots,n;$
\item[(iii)] $\cK_T=\bigvee_{\alpha\in \FF^+_n} V_\alpha \cH.$
\end{enumerate}
The isometric dilation theorem for row contractions (see
\cite{Po-isometric}, \cite{Bu}, \cite{F})
 asserts that
  every  row contraction $T$ has a minimal isometric
  dilation $V$ which is uniquely
determined up to an isomorphism.

The main result of this section is the following.
\begin{theorem}
\label{equivalent2} Let $A:=[A_1,\ldots, A_n]$ and $B:=[B_1,\ldots,
B_n]$ be in $[B(\cH)^n]_1^-$ and let $V:=[V_1,\ldots, V_n]$ on
$\cK_A\supseteq \cH$ and $W:=[W_1,\ldots, W_n]$ on $\cK_B\supseteq
\cH$ be the minimal isometric dilations of $A$ and $B$,
respectively. If
 $c\geq1$, then the following statements are equivalent:
\begin{enumerate}
\item[(i)] $A\overset{H}{{\underset{c}\prec}}\, B$;
\item[(ii)]
there is an operator $L_{B,A}\in B(\cK_B, \cK_A)$ with
$\|L_{B,A}\|\leq c$ such that $L_{B,A}|_\cH=I_\cH$ and
\begin{equation*}
 L_{B,A} W_i=V_iL_{B,A},\qquad i=1,\ldots,n.
\end{equation*}
\end{enumerate}

\end{theorem}
\begin{proof}
 Assume that (i) holds. According to Theorem \ref{equivalent}, we have
$K_A\leq c^2 K_B$, where $K_X$ is the multi-Toeplitz kernel
associated with $X\in [B(\cH)^n]_1^-$.

 Since $V:=[V_1,\ldots, V_n]$ is the minimal isometric
dilation of $A:=[A_1,\ldots, A_n]$, we have
$\cK_A=\bigvee_{\alpha\in \FF_n^+} V_\alpha \cH$ and
$V_\alpha^*|_{\cH}=A_\alpha^*$ for any $\alpha\in \FF_n^+$. Similar
properties hold for $W$ and $B$. Hence,  and taking into account
that $V_1,\ldots, V_n$ and $W_1,\ldots, W_n$ are isometries with
orthogonal ranges, respectively,  we have
\begin{equation*}
\begin{split}
\left\|\sum_{|\alpha|\leq m} V_\alpha h_\alpha\right\|^2 &=
\sum_{\alpha>_l \beta, |\alpha|, |\beta|\leq
m}\left<V_{\alpha\backslash _l \beta} h_\alpha, h_\beta\right>
+\sum_{|\alpha|\leq m} \left<h_\alpha, h_\alpha\right> +
\sum_{\beta>_l \alpha, |\alpha|, |\beta|\leq
m}\left<V_{\beta\backslash _l \alpha}^* h_\alpha, h_\beta\right>
\\
&= \sum_{\alpha>_l \beta, |\alpha|, |\beta|\leq
m}\left<A_{\alpha\backslash _l \beta} h_\alpha, h_\beta\right>
+\sum_{|\alpha|\leq m} \left<h_\alpha, h_\alpha\right> +
\sum_{\beta>_l \alpha, |\alpha|, |\beta|\leq
m}\left<A_{\beta\backslash _l \alpha}^* h_\alpha, h_\beta\right>
\\
&=\sum_{|\alpha|\leq m, |\beta|\leq m} \left<K_A(\beta, \alpha)
h_\alpha, h_\beta\right> =
\left<\left[K_A(\beta,\alpha)\right]_{|\alpha|, |\beta|\leq m}{\bf
h}_m, {\bf h}_m\right>
\end{split}
\end{equation*}
for any $m\in \NN$  and ${\bf h}_m\in \oplus_{|\alpha|\leq m}
\cH_{\alpha}$, where each $\cH_{\alpha}$ is a copy of $\cH$.
Similarly, one can show that
$$
\left\|\sum_{|\alpha|\leq m} W_\alpha h_\alpha\right\|^2  =
\left<\left[K_B(\beta, \alpha)\right]_{|\alpha|, |\beta|\leq m}{\bf
h}_m, {\bf h}_m\right>.
$$
Since $K_A\leq c^2 K_B$, we deduce that
$$
\left\|\sum_{|\alpha|\leq m} V_\alpha h_\alpha\right\|\leq c
\left\|\sum_{|\alpha|\leq m} W_\alpha h_\alpha\right\|.
$$
Consequently, we can define an  operator $L_{B,A}:\cK_B\to \cK_A$ by
setting
\begin{equation}
\label{LBA} L_{B,A}\left(\sum_{|\alpha|\leq m} W_\alpha
h_\alpha\right):=\sum_{|\alpha|\leq m} V_\alpha h_\alpha
\end{equation}
for any $m\in \NN$ and $h_\alpha\in \cH$, $\alpha\in \FF_n^+$. Due
to the considerations above, $L_{B,A}$ is a bounded  operator with
$\|L_{B,A}\|\leq c$. Note that since $L_{B,A}|_\cH=I_\cH$, we have
$\|L_{B,A}\|\geq 1$.  One can easily see that  $L_{B,A}
W_i=V_iL_{B,A}$ for $ i=1,\ldots,n$. Therefore, (ii) holds.

Conversely, assume that there is an operator $L_{B,A}\in B(\cK_B,
\cK_A)$ with norm $\|L_{B,A}\|\leq c$  such that
$L_{B,A}|_\cH=I_\cH$ and $ L_{B,A} W_i=V_iL_{B,A}$, $ i=1,\ldots,n$.
Hence, we deduce that $ L_{B,A}\left(\sum_{|\alpha|\leq m} W_\alpha
h_\alpha\right)=\sum_{|\alpha|\leq m} V_\alpha h_\alpha$ for any
$m\in \NN$ and $h_\alpha\in \cH$, $\alpha\in \FF_n^+$. Since
$\|L_{B,A}\|\leq c$ , we have
$$
\left\|\sum_{|\alpha|\leq m} V_\alpha h_\alpha\right\|^2 \leq
c^2\left\|\sum_{|\alpha|\leq m} W_\alpha h_\alpha\right\|^2,
$$
which is equivalent to the inequality
$$
\left<\left[K_A(\beta, \alpha)\right]_{|\alpha|, |\beta|\leq m}{\bf
h}_m, {\bf h}_m\right> \leq c^2 \left<\left[K_B(\beta,
\alpha)\right]_{|\alpha|, |\beta|\leq m}{\bf h}_m, {\bf h}_m\right>
$$
for any $m\in \NN$  and ${\bf h}_m\in \oplus_{|\alpha| \leq m}
\cH_{\alpha }$. Hence, we obtain  $K_A\leq c^2 K_B$. Using again
Theorem \ref{equivalent}, we deduce item (i). The proof is complete.
\end{proof}

Using Theorem \ref{equivalent2} and relation \eqref{equi}, we deduce
the following consequences.
\begin{corollary}\label{LBA-inf} If  $A, B\in [B(\cH)^n]_1^-$ and
 $A\overset{H}{{\prec}}\, B$, then
 \begin{equation*}
 \begin{split}
 \|L_{B,A}\|&=\inf\{c\geq 1:\ A\overset{H}{{\underset{c}\prec}}\,
 B\}\\
 &=\inf\{c\geq 1:\  P( rA, R)\leq c^2 P(rB, R)\quad \text{ for any }
\quad r\in [0,1)\}.
 \end{split}
 \end{equation*}
\end{corollary}

\begin{corollary}
\label{Lr}
  Let   $A:=[A_1,\ldots, A_n]$ and $B:=[B_1,\ldots,
B_n]$  be    in $[B(\cH)^n]_1^-$.  Then
 $A\overset{H}{{\prec}}\, B$ if and only if \
    $\sup_{r\in [0, 1)} \|L_{rA,rB}\|<\infty$.
 Moreover,
in this case,
$$\|L_{A,B}\|=\sup_{r\in [0, 1)} \|L_{rA,rB}\|
$$
and $r\mapsto  \|L_{rA,rB}\|$ is increasing on $[0,1)$.
\end{corollary}

\begin{proof} Assume that $A\overset{H}{{\prec}}\, B$.  Then, due to
Theorem \ref{equivalent2}, $A\overset{H}{{\underset{c}\prec}}\, B$
if and only if there is an operator $L_{B,A}\in B(\cK_B, \cK_A)$
with $\|L_{B,A}\|\leq c$ such that $L_{B,A}|_\cH=I_\cH$ and
$
 L_{B,A} W_i=V_iL_{B,A}$, $ i=1,\ldots,n.
$
Consequently, taking $c=\|L_{B,A}\|$, we deduce that
$A\overset{H}{{\underset{\|L_{B,A}\|}\prec}}\, B$, which due to
\eqref{equi}, is equivalent to
$$
P(rA,R)\leq \|L_{B,A}\|^2 P(rB,R)
$$
for any $r\in [0,1)$. Using again relation \eqref{equi}, we have
$tA\overset{H}{{\underset{\|L_{B,A}\|}\prec}}\, tB$ for any $t\in
[0,1)$. Applying Theorem \ref{equivalent2} to the operators $tA$ and
$tB$, we deduce that $\|L_{tA,tB}\|\leq \|L_{B,A}\|$.

Conversely, suppose that $c:=\sup_{r\in [0, 1)}
\|L_{rA,rB}\|<\infty$. Since $\|L_{rA,rB}\|\leq c$, Theorem
\ref{equivalent2} implies $rA\overset{H}{{\underset{c}\prec}}\, rB$
for any $r\in [0,1)$. Due to \eqref{equi}, we have $P(rtA,R)\leq c^2
P(rtB,R)$ for any $t,r\in [0,1)$. Hence,
$A\overset{H}{{\underset{c}\prec}}\, B$ which, due to Theorem
\ref{equivalent2}, implies $\|L_{B,A}\|\leq c$. Therefore,
$\|L_{A,B}\|=\sup_{r\in [0, 1)} \|L_{rA,rB}\|$. The fact that
$r\mapsto  \|L_{rA,rB}\|$ is  an increasing function on $[0,1)$
follows  from the latter relation. This completes the proof.
\end{proof}

The intertwining operator $L_{B,A}$ introduced in this section will
play a very important role in this paper.

\bigskip

\section{The reconstruction operator and Harnack domination}

 In this section,  we obtain some  results  concerning
 the reconstruction operator and Harnack equivalence  on $[B(\cH)^n]_1$.
These results are needed in the next sections.

We need to recall from \cite{Po-poisson} a few basic facts about
noncommutative Poisson transforms on Cuntz-Toeplitz algebras. Let
$T:=[T_1,\dots, T_n$ be   row contraction, i.e.,
$$
T_1T_1^*+\cdots +T_nT_n^*\leq I_\cH.
$$
 We define the defect operator
$\Delta_{T,r}:=(I_\cH-r^2T_1T_1^*-\cdots -r^2 T_nT_n^*)^{1/2}$ for
each $0< r\leq 1$.
The noncommutative Poisson  kernel associated with $T$ is the family
of operators
$$
K_{T,r} :\cH\to  \overline{\Delta_{T,r}\cH} \otimes  F^2(H_n), \quad
0< r\leq 1,
$$
defined by
\begin{equation*}
K_{T,r}h:= \sum_{k=0}^\infty \sum_{|\alpha|=k} r^{|\alpha|}
\Delta_{T,r} T_\alpha^*h\otimes  e_\alpha,\quad h\in \cH.
\end{equation*}
We recall that $\{e_\alpha\}_{\alpha\in \FF_n^+}$ is an orthonormal
basis for $F^2(H_n)$. When $r=1$, we denote $\Delta_T:=\Delta_{T,1}$
and $K_T:=K_{T,1}$. The operators $K_{T,r}$ are isometries if
$0<r<1$, and
$$
K_T^*K_T=I_\cH- \text{\rm SOT-}\lim_{k\to\infty} \sum_{|\alpha|=k}
T_\alpha T_\alpha^*.
$$
Thus $K_T$ is an isometry if and only if $T$ is a {\it pure} row
 contraction,
i.e., $ \text{\rm SOT-}\lim\limits_{k\to\infty} \sum_{|\alpha|=k}
T_\alpha T_\alpha^*=0.$ We denote by $C^*(S_1,\ldots, S_n)$ the
Cuntz-Toeplitz $C^*$-algebra generated by the left creation
operators (see \cite{Cu}). The noncommutative Poisson transform at
 $T:=[T_1,\ldots, T_n]\in [B(\cH)^n]_1^-$ is the unital completely contractive  linear map
 $P_T:C^*(S_1,\ldots, S_n)\to B(\cH)$ defined by
 \begin{equation*}
 P_T[f]:=\lim_{r\to 1} K_{T,r}^* (I_\cH \otimes f)K_{T,r}, \qquad f\in C^*(S_1,\ldots,
 S_n),
\end{equation*}
 where the limit exists in the norm topology of $B(\cH)$. Moreover, we have
 $$
 P_T[S_\alpha S_\beta^*]=T_\alpha T_\beta^*, \qquad \alpha,\beta\in \FF_n^+.
 $$
 When $T:=[T_1,\ldots, T_n]$  is a pure row contraction,
   we have $$P_T[f]=K_T^*(I_{\cD_{T}}\otimes f)K_T,
   $$
   where $\cD_T=\overline{\Delta_T \cH}$.
We refer to \cite{Po-poisson}, \cite{Po-curvature},  and
\cite{Po-unitary} for more on noncommutative Poisson transforms on
$C^*$-algebras generated by isometries.

Consider now  the particular case when $n=1$. In this case,   the
free pluriharmonic Poisson kernel $P(Y,R)$ coincides with
$$
Q(Y, U):=\sum_{k=1} {Y^*}^k\otimes U^k + I+ \sum_{k=1}^\infty
 Y^k\otimes {U^*}^k,\qquad \|Y\|<1,
$$
where the convergence of the series is in the operator norm topology
and $U$ is the  unilateral shift acting on the Hardy space
$H^2(\TT)$. For each contraction $T\in B(\cH)$, consider the
operator-valued Poisson kernel defined by
$$
K(z,T):= \sum_{k=1}^\infty z^k {T^*}^k+I+\sum_{k=1}^\infty \bar z^k
T^k,\qquad z\in \DD.
$$
Using Theorem \ref{equivalent} from \cite{Po-hyperbolic}, we can
deduce the following.
\begin{proposition}\label{n=1}
Let $ T$ and $ T'$ be in $[B(\cH)]_1^-$ and let $c\geq 1$. Then the
following statements are equivalent:
\begin{enumerate}
\item[(i)] $T\overset{H}{{\underset{c}\prec}}\, T'$;
\item[(ii)] $Q(rT,U)\leq c^2 Q(rT', U)$ for any $r\in [0,1)$;
\item[(iii)] $K(z,T)\leq c^2K(z,T')$ for any $z\in \DD$.
\end{enumerate}
\end{proposition}

\begin{proof} The equivalence $(i)\leftrightarrow (ii)$ follows from
Theorem \ref{equivalent}  of \cite{Po-hyperbolic}, in the particular
case  when $n=1$. To prove the  implication $(ii)\implies (iii)$, we
apply the noncommutative Poisson transform  (when $n=1$) at
$e^{it}I$
 to the inequality $(ii)$. We have
 $$
 K(re^{it}, T)=(\text{\rm id}\otimes P_{e^{it}I} )[Q(rT, U)]\leq
 c^2( \text{\rm id}\otimes P_{e^{it}I} )[Q(rT', U)]= c^2 K(re^{it}, T')
 $$
for any $r\in[0,1)$ and $t\in \RR$. It remains to prove that
$(iii)\implies (ii)$. Since
$$
\left<( {T^*}^k\otimes U^k)( h_m \otimes e^{imt}), h_p\otimes
e^{ipt}\right>_{\cH\otimes H^2(\TT)}= \frac{1}{2\pi} \int_{-\pi}^\pi
\left<e^{ikt} {T^*}^k(e^{imt} h_m), e^{ipt} h_p\right>_\cH dt
$$
for any $h_m, h_p\in \cH$  and  $k,m,p\in \NN$, we deduce that
\begin{equation*}
\left<\left(c^2 Q(rT',U)-Q(rT,U)\right) h(e^{it}),
h(e^{it})\right>_{\cH\otimes H^2(\TT)}= \frac{1}{2\pi}
\int_{-\pi}^\pi
\left<\left(c^2K(re^{it},T')-K(re^{it},T)\right)h(e^{it}),
h(e^{it})\right>_{\cH}
\end{equation*}
for any function $e^{it}\mapsto h(e^{it})$ in $\cH\otimes H^2(\TT)$.
Now, it is clear that $(iii)\implies (ii)$. The proof is complete.
\end{proof}

The next result makes an interesting connection between Harnack
domination in $[B(\cH)^n]_1^-$ and reconstruction operators.

\begin{theorem}
\label{equivalent4} Let $A:=[A_1,\ldots, A_n]$ and $B:=[B_1,\ldots,
B_n]$ be in $[B(\cH)^n]_1^-$ and let $c>0$. Then the following
statements are equivalent:
\begin{enumerate}
\item[(i)]
$A\overset{H}{{\underset{c}\prec}}\, B$;
\item[(ii)] $R_A\overset{H}{{\underset{c}\prec}}\, R_B$, where
$R_X:=  X_1^*\otimes R_1+\cdots +   X_n^*\otimes R_n$ is the
reconstruction operator  associated with $X:=(X_1,\ldots, X_n)\in
[B(\cH)^n]_1^-$ and the right creation operators $ R_1,\ldots, R_n$.
\item[(iii)] $R_A^*\overset{H}{{\underset{c}\prec}}\, R_B^*$;
\end{enumerate}
\end{theorem}
\begin{proof}
Assume that (i) holds. According to Theorem \ref{equivalent} from
\cite{Po-hyperbolic}, we have
\begin{equation}
\label{P<P} P(rA, S)\leq c^2P(rB,S)
\end{equation}
for any $r\in [0,1)$, where $S:=(S_1,\ldots, S_n)$ is the $n$-tuple
of left creation operators. Let $U$ be the unilateral shift on the
Hardy space $H^2(\TT)$. Since $R_i^* R_j=\delta_{ij} I$, it is easy
to see that $[U^*\otimes R_1,\ldots, U^*\otimes R_n]$ is a row
contraction acting from $[H^2(\TT)\otimes F^2(H_n)]^{(n)}$ to
$H^2(\TT)\otimes F^2(H_n)$. Using the noncommutative Poisson
transform  at $[U^*\otimes R_1,\ldots, U^*\otimes R_n]$ and
inequality \eqref{P<P}, we have
\begin{equation*}
\begin{split}
Q(rR_A,U)&= \left(
\text{\rm id}\otimes P_{[U^*\otimes R_1,\ldots, U^*\otimes R_n]}\right)[P(rA,S)]\\
&\leq c^2\left(
\text{\rm id}\otimes P_{[U^*\otimes R_1,\ldots, U^*\otimes R_n]}\right)[P(rB,S)]\\
&=c^2Q(rR_B,U)
\end{split}
\end{equation*}
for any $r\in [0,1)$. Applying Proposition \ref{n=1}, we deduce that
$R_A\overset{H}{{\underset{c}\prec}}\, R_B$. Conversely, assume that
(ii) holds. Using again Proposition \ref{n=1}, we have
\begin{equation}
\label{K<K}
 K(re^{it},R_A)\leq c^2K(re^{it},R_B),\qquad r\in [0,1),
t\in \RR.
\end{equation}
Taking $t=0$, we deduce that $P( rA, R)\leq c^2 P( rB,R)$ for any
$r\in [0,1)$, which   implies $A\overset{H}{{\underset{c}\prec}}\,
B$. The equivalence $(ii)\leftrightarrow (iii)$ is due to
Proposition \ref{n=1} and the fact that \eqref{K<K} is equivalent to
\begin{equation*}
 K(re^{it},R_A^*)\leq c^2K(re^{it},R_B^*),\qquad r\in [0,1),
t\in \RR.
\end{equation*}
The proof is complete.
\end{proof}

\bigskip

\section{Harnack type equivalence on $[B(\cH)^n]_1^-$ and dilations}

In this section, using  noncommutative dilation theory for row
contractions (\cite{Po-isometric}, \cite{Po-charact}), we obtain
several results concerning the Harnack
 parts of $[B(\cH)^n]_1^-$.

 Since  the Harnack type domination $\overset{H}{\prec} $
is a preorder relation on $[B(\cH)^n]_1^-$, it induces an equivalent
relation $\overset{H}\sim$ on $[B(\cH)^n]_1^-$, which we call
Harnack equivalence. The equivalence classes with respect to
$\overset{H}\sim$ are called Harnarck parts of $[B(\cH)^n]_1^-$.  It
is easy to see that $A$ and $B$ are Harnack equivalent (denote
$A\overset{H}{\sim}\, B$)  if and only if there exists $c\geq 1$
such that
\begin{equation*}
 \frac{1}{c^2}\text{\rm Re}\,p(B_1,\ldots, B_n)\leq
\text{\rm Re}\,p(A_1,\ldots, A_n)\leq c^2 \text{\rm
Re}\,p(B_1,\ldots, B_n)
\end{equation*}
for any noncommutative polynomial with matrix-valued coefficients
$p\in \CC[X_1,\ldots, X_n]\otimes M_{m}$, $m\in \NN$, such that
$\text{\rm Re}\,p\geq 0$. We also use the notation
$A\overset{H}{{\underset{c}\sim}}\, B$ \ if \
$A\overset{H}{{\underset{c}\prec}}\, B$ \ and \
$B\overset{H}{{\underset{c}\prec}}\, A$.

First we consider a few characterizations for Harnack domination in
$[B(\cH)^n]_1^-$ (other characterizations were considered in
\cite{Po-hyperbolic}).

\begin{theorem}
\label{Harnack} Let $A:=[A_1,\ldots, A_n]$ and $B:=[B_1,\ldots,
B_n]$ be in $[B(\cH)^n]_1^-$ and let $c\geq 1$.  Then the following
statements are equivalent:
\begin{enumerate}
\item[(i)] $A\overset{H}{{\underset{c}\sim}}\, B$;
\item[(ii)] $\frac{1}{c^2} K_B\leq K_A\leq c^2 K_B,$
 where $K_X$ is the
multi-Toeplitz kernel associated with $X\in [B(\cH)^n]_1^-$;
\item[(iii)] $R_A\overset{H}{{\underset{c}\sim}}\, R_B$, where $R_X$ is the reconstruction
operator associated with $X\in [B(\cH)^n]_1^-$;
\item[(iv)] $L_{B,A}$ is an invertible operator  with
$\|L_{B,A}\|\leq c$ and $\|L_{B,A}^{-1}\|\leq c$.
\end{enumerate}
\end{theorem}
\begin{proof}
The equivalences $(i)\leftrightarrow (ii)$  and  $(i)\leftrightarrow
(iv)$ are  due to  Theorem \ref{equivalent} and  Theorem
\ref{equivalent2}, respectively. Note also that the equivalence
$(i)\leftrightarrow (iii)$  follows from  Theorem \ref{equivalent4}.
\end{proof}

Let us recall \cite{Po-isometric} the Wold type decomposition for
sequences of isometries with orthogonal ranges. Let $V:=[V_1,\ldots,
V_n]$, $V_i\in B(\cK)$, be such that $V_i^*V_j=\delta_{ij}I$. Then
$\cK$ decomposes into an orthogonal sum $\cK=\cG_V\oplus M_+(\cL_V)$
such that $\cG_V$ and $M_+(\cL_V)$ reduce each operator $V_i$,
$i=1,\ldots, n$, and such that $\left(I_\cK-\sum_{i=1}^n V_i
V_i^*\right)|_{\cG_V}=0$ and $\left(V_1|_{M_+(\cL_V)},\ldots,
V_n|_{M_+(\cL_V)}\right)$ is unitarilly equivalent to the $n$-tuple
of left creation operators $(S_1\otimes I_{\cL_V},\ldots, S_n\otimes
I_{\cL_V})$. Moreover, the decomposition is unique and we have
$$
\cG_V=\bigcap_{k=0}^\infty \left(\bigoplus_{|\alpha|=k}V_\alpha
\cK_V\right)\quad \text{ and } \quad M_+(\cL_V)=\bigoplus_{\alpha\in
\FF_n^+} V_\alpha \cL_V,
$$
where $\cL_V:=\cK\ominus\left(\oplus_{i=1}^n V_i\cK\right)$.

\begin{theorem}
\label{equivalent3} Let $A:=[A_1,\ldots, A_n]$ and $B:=[B_1,\ldots,
B_n]$ be in $[B(\cH)^n]_1^-$ and let $V:=[V_1,\ldots, V_n]$ on
$\cK_A\supseteq \cH$ and $W:=[W_1,\ldots, W_n]$ on $\cK_B\supseteq
\cH$ be the minimal isometric dilations of $A$ and $B$,
respectively. Then the following statements
 hold.
\begin{enumerate}
\item[(i)]
If $A\overset{H}{\sim}\, B$, then $V$ and $W$ are unitarily
equivalent.
\item[(ii)]
 Let $\cK_V=\cG_V\oplus M(\cL_V)$ and
$\cK_W=\cG_W\oplus M(\cL_W)$ be   the Wold type decompositions of
$V$ and $W$, respectively, and  let $L_{B,A}$ be the operator
defined by  \eqref{LBA}. If $A\overset{H}{\sim}\, B$, then
 $$
  L_{B,A}^* (\cL_V)=\cL_W \quad \text{ and }  \quad
  L_{B,A}(\cG_W)=\cG_V.
  $$
  \end{enumerate}
  \end{theorem}
\begin{proof} To
prove part (i), note that, due to Theorem \ref{equivalent2} and
Theorem \ref{Harnack}, $L_{B,A}$ is an invertible operator such that
$ L_{B,A} W_i=V_iL_{B,A}$, $ i=1,\ldots,n$. Applying Theorem 2.1
from \cite{Po-funct}, we deduce that $V$ and $W$ are unitarilly
equivalent. Part (ii) is a consequence of  Theorem \ref{Harnack},
part (iv), and the proof of Theorem 2.1 from \cite{Po-funct}.
\end{proof}

\begin{corollary}
Let $A:=[A_1,\ldots, A_n]$ and $B:=[B_1,\ldots, B_n]$ be in
$[B(\cH)^n]_1^-$ such that $A\overset{H}{\sim}\, B$.
  Then the following
  properties hold.
  \begin{enumerate}
 \item[(i)]
$A$ is a pure   row contraction if and only if $B$ has the same
property.
\item[(ii)] $A_1 A_1^*+\cdots + A_nA_n^*=I$  if and only if $B_1 B_1^*+\cdots +
B_nB_n^*=I$.
\item[(iii)] $\text{\rm rank\,}\Delta_A=\text{\rm rank\,} \Delta_B$.
\end{enumerate}
\end{corollary}

\begin{proof}According to  \cite{Po-isometric}, $A$ is a pure row
contraction if and only if its minimal isometric dilation $V$ is a
pure row isometry, i.e., $\cK_V=M_+(\cL_V)$.  In this case we have
$\cG_V=\{0\}$ and, since $L_{B,A}$ is invertible, part (ii) of
Theorem \ref{equivalent3} implies $\cG_W=\{0\}$. Therefore,
$\cK_W=M_+(\cL_W)$, which shows that $B$ is a pure row contraction.
To prove (ii), note that, due to \cite{Po-isometric}, $A_1
A_1^*+\cdots + A_nA_n^*=I$ is and only if $V_1 V_1^*+\cdots +
V_nV_n^*=I$. The Wold type decomposition for $V$ shows that
$\cK_V=\cG_V$. Since $L_{B,A}(\cG_W)=\cG_V$ and $L_{B,A}:\cK_W\to
\cK_V$ is an invertible operator, part (ii) of Theorem
\ref{equivalent3} implies $\cG_W=\cK_W$. Therefore, we have $W_1
W_1^*+\cdots W_nW_n^*=I$. Using again \cite{Po-isometric}, we deduce
that  $B_1 B_1^*+\cdots + B_nB_n^*=I$.

The geometric structure of the minimal isometric dilations of row
contractions  (see \cite{Po-isometric}) implies that
$$
\dim \cD_A=\dim \cL_V \quad \text{ and } \quad \dim \cD_B=\dim
\cL_W.
$$
Since $L_{B,A}$ is invertible and $L_{B,A}^*(\cL_V)=\cL_W$, we
deduce that $\dim \cL_V=\dim \cL_W$, which completes the proof of
part (iii).
\end{proof}

\begin{corollary}  If  $A\in [B(\cH)^n]_1^-$
 is a row isometry (or co-isometry),
    then
 the Harnack part of $A$ reduces to $\{A\}$.
\end{corollary}
\begin{proof}
 Let $A $ and $B $ be in $[B(\cH)^n]_1^-$ such that $A\overset{H}{\sim}\, B$.
  First, assume that $A$ is a row
 isometry and
  let $W:=[W_1,\ldots, W_n]$ on $\cK_B\supseteq \cH$
be the minimal isometric dilations of  $B$. According to Theorem
\ref{equivalent2}, $L_{B,A}\in B(\cK_B, \cK_A)$  is an invertible
operator   such that $L_{B,A}|_\cH=I_\cH$ and
\begin{equation*}
 L_{B,A} W_i=V_iL_{B,A},\qquad i=1,\ldots,n.
\end{equation*}
Since $A$ is a row isometry, we have $\cK_A=\cH$ and $V=A$.  Now,
one can easily see that $\cK_B=\cH$,  and $W_i=B_i$, $i=1,\ldots,
n$. Hence $L_{B,A}=I_\cH$ and the intertwining  relation above
implies $W_i=A_i$, $i=1,\ldots, n$. Consequently, $A=B$.

Now, assume that $A$ is a co-isometry, i.e., $A_1A_1^*+\cdots +
A_nA_n^*=I$. Setting $R_A:=\sum_{i=1}^n A_i^*\otimes R_i$, we have
$R_A^* R_A=I$. On the other hand, due to Theorem \ref{Harnack},
$A\overset{H}{\sim}\, B$ if and only if $R_A\overset{H}{\sim}\,
R_B$. Applying  the first part of this corollary when $n=1$ and $A$
and $B$ are replaced by $R_A$ and $R_B$, respectively, we deduce
that $R_A=R_B$. Consequently, $A=B$, which completes the proof.
\end{proof}

The characteristic  function associated with an arbitrary row
contraction $T:=[T_1,\ldots, T_n]$, \ $T_i\in B(\cH)$, was
introduced in \cite{Po-charact} (see \cite{SzF-book} for the
classical case $n=1$) and it was proved to be  a complete unitary
invariant for completely non-coisometric  (c.n.c.) row contractions.
The characteristic function  of $T$ is  a   multi-analytic operator
with respect to $S_1,\ldots, S_n$,
$$
\tilde{\Theta}_T:F^2(H_n)\otimes \cD_{T^*}\to F^2(H_n)\otimes \cD_T,
$$
with the formal Fourier representation
\begin{equation*}
\begin{split}
 \Theta_T(R_1,\ldots, R_n):= -I_{F^2(H_n)}\otimes T+
\left(I_{F^2(H_n)}\otimes \Delta_T\right)&\left(I_{F^2(H_n)\otimes
\cH}
-\sum_{i=1}^n R_i\otimes T_i^*\right)^{-1}\\
&\left[R_1\otimes I_\cH,\ldots, R_n\otimes I_\cH \right]
\left(I_{F^2(H_n)}\otimes \Delta_{T^*}\right),
\end{split}
\end{equation*}
where $R_1,\ldots, R_n$ are the right creation operators on the full
Fock space $F^2(H_n)$.
 More precisely, we have
$$
\tilde\Theta_T=\text{\rm SOT-}\lim_{r\to 1}\Theta_T (rR_1,\ldots,
rR_n).
$$
For definitions and basic facts concerning multi-analytic operators
on Fock spaces we refer to Section 5 and
\cite{Po-charact}--\cite{Po-analytic}.

 We obtained in
\cite{Po-charact} a functional model for c.n.c row contractions. In
the particular case when $T$ is a pure row contraction, we proved
that $T$ is unitarilly equivalent to the model row contraction
$$\TT:=(P_{\HH}(S_1\otimes
I_{\cD_T})|_{\HH},\ldots,P_{\HH}(S_n\otimes I_{\cD_T})|_{\HH}),$$
where $P_\HH$ is the orthogonal projection  on
$$\HH_T:=[F^2(H_n)\otimes \cD_T]\ominus \tilde \Theta_T (F^2(H_n)\otimes
\cD_{T^*})
$$
and $S_1,\ldots, S_n$ are the left creation operators on the full
Fock space $F^2(H_n)$. In this case, the minimal isometric dilation
of $\TT$ is $[S_1\otimes I_{\cD_T},\ldots, S_n\otimes I_{\cD_T}]$.

 Now, we can show that there is a strong connection between the operator $L_{B,A}$, which
  is essentially a multi-analytic operator, and
  the characteristic functions of $A$ and $B$.

\begin{theorem}
\label{analytic} Let $A:=[A_1,\ldots, A_n]$ and $B:=[B_1,\ldots,
B_n]$ be  pure row contractions with the property that
$A\overset{H}{\sim}\,B$. Let $\tilde\Theta_A: F^2(H_n)\otimes
\cD_{A^*}\to F^2(H_n)\otimes \cD_{A}$ and $\tilde\Theta_B:
F^2(H_n)\otimes \cD_{B^*}\to F^2(H_n)\otimes \cD_{B}$ be the
characteristic functions of $A$ and $B$, respectively. Then there is
an invertible multi-analytic operators $L:F^2(H_n)\otimes \cD_{B}\to
F^2(H_n)\otimes \cD_{A}$ such that
$$
L^*\tilde\Theta_A[F^2(H_n)\otimes
\cD_{A^*}]=\tilde\Theta_B[F^2(H_n)\otimes \cD_{B^*}].
$$
\end{theorem}

\begin{proof}
Let $A:=[A_1,\ldots, A_n]$ and $B:=[B_1,\ldots, B_n]$ be in
$[B(\cH)^n]_1^-$ and let $V:=[V_1,\ldots, V_n]$ on $\cK_A\supseteq
\cH$ and $W:=[W_1,\ldots, W_n]$ on $\cK_B\supseteq \cH$ be the
minimal isometric dilations of $A$ and $B$, respectively. Since $A$
and $B$ are pure row contraction we have $\cK_A=M_+(\cL_V)$ and
$\cK_B=M_+(\cL_W)$. According to Theorem \ref{equivalent2},
$L_{B,A}:M_+(\cL_W)\to M_+(\cL_V)$
  is an invertible operator such
that $L_{B,A}|_\cH=I_\cH$ and
\begin{equation*}
 L_{B,A} W_i=V_iL_{B,A},\qquad i=1,\ldots,n.
\end{equation*}
We recall that the {\it Fourier transform} $\Phi_V: M_+(\cL_V)\to
F^2(H_n)\otimes \cL_V$ is a unitary operator defined by
$\Phi_V(V_\alpha \ell):=e_\alpha\otimes \ell$ for any $ \ell\in
\cL_V$, $\alpha\in \FF_n^+$. Note that $\Phi_V V_i=(S_i\otimes
I_{\cL_V}) \Phi_V$, $i=1,\ldots, n$. Similarly one can define the
Fourier transform $\Phi_W$. Note that the operator
$$
\Gamma:=\Phi_V L_{B,A} \Phi_W^*: F^2(H_n)\otimes \cL_W\to
F^2(H_n)\otimes \cL_V
$$
has the property that $\Gamma (S_i\otimes I_{\cL_W})=(S_i\otimes
I_{\cL_V}) \Gamma$ for $i=1,\ldots, n$. Now, using the
identification of $\cL_V$ and $\cL_W$ with $\cD_A$ and $\cD_B$,
respectively (see \cite{Po-charact}), we deduce that there is an
invertible multi-analytic operator $L:F^2(H_n)\otimes \cD_{B}\to
F^2(H_n)\otimes \cD_{A}$ such that $L(\HH_B)=\HH_A$. Consequently,
we have
$$
L^*\tilde\Theta_A[F^2(H_n)\otimes
\cD_{A^*}]=\tilde\Theta_B[F^2(H_n)\otimes \cD_{B^*}],
$$
which completes the proof.
\end{proof}

\bigskip

\section{The operator $L_{B,A}$ and the hyperbolic  distance on  the Harnack parts of
$[B(\cH)^n]_1^-$}

 In this section we  express the
hyperbolic distance  $\delta$ in terms of the intertwining operator
$L_{B,A}$ and
 obtain  an explicit formula for the norm  $L_{B,A}$ in terms of the reconstruction
operators $R_A$ and $R_B$. We also show that $\|L_{B,A}\|$ is
invariant under the automorphism group $Aut([B(\cH)^n]_1)$.

 In \cite{Po-hyperbolic}, we introduced a hyperbolic  ({\it
Poincar\'e-Bergman} type)   metric $\delta$ on the Harnack parts of
$[B(\cH)^n]_1^-$ as follows. If $A$ and $B$ are Harnack equivalent
in $[B(\cH)^n]_1^-$, we define
\begin{equation}
\label{hyperbolic} \delta(A,B):=\ln \omega(A,B),
\end{equation}
 where
\begin{equation*}
 \omega(A,B):=\inf\left\{ c\geq 1: \
A\overset{H}{{\underset{c}\sim}}\, B   \right\}.
\end{equation*}

Basic properties of the hyperbolic metric $\delta$ were considered
in \cite{Po-hyperbolic}. Now we establish new connections with
multivariable operator theory.

\begin{lemma}
\label{OMr}
  Let   $A:=[A_1,\ldots, A_n]$ and $B:=[B_1,\ldots,
B_n]$ be   in $[B(\cH)^n]_1^-$. Then $A\overset{H}{\sim}\, B$ if and
only if the operator $L_{B,A}$ is invertible. In this case,
$L_{B,A}^{-1}=L_{A,B}$ and
$$
\omega(A,B)= \max\left\{ \|L_{A,B}\|, \|L_{B,A}\|\right\}.
$$
\end{lemma}

\begin{proof} The first part of this lemma is a simple consequence
of Theorem \ref{equivalent2}. To prove the last part, assume that
$A\overset{H}{{\underset{c}\sim}}\, B $ for some $c\geq 1$. Due to
Theorem \ref{equivalent2}, we have $\|L_{B,A}\|\leq c$ and
$\|L_{A,B}\|\leq c$. Consequently,
\begin{equation}\label{Max}
\max\left\{ \|L_{A,B}\|, \|L_{B,A}\|\right\}\leq \inf\left\{ c\geq
1: \ A\overset{H}{{\underset{c}\sim}}\, B   \right\}=\omega(A,B).
\end{equation}

On the other hand, set $c_0:=\|L_{B,A}\|$ and $c_0':=\|L_{A,B}\|$.
Using again Theorem \ref{equivalent2}, we deduce that
$A\overset{H}{{\underset{c_0}\prec}}\, B$ and
$B\overset{H}{{\underset{c_0'}\prec}}\, A$. Hence, we deduce that
$A\overset{H}{{\underset{d}\sim}}\, B $, where
$d:=\max\{c_0,c_0'\}$. Consequently, $\omega(A,B)\leq d$, which
together with relation \eqref{Max} imply $\omega(A,B)= \max\left\{
\|L_{A,B}\|, \|L_{B,A}\|\right\}$. This completes the proof.
\end{proof}

Using Lemma \ref{OMr}, we can express the hyperbolic metric in terms
of the operator $L_{A,B}$.

\begin{theorem}\label{formula}
Let   $A:=[A_1,\ldots, A_n]$ and $B:=[B_1,\ldots, B_n]$ be   in
$[B(\cH)^n]_1^-$ such that $A\overset{H}{\sim}\, B$. Then
  the metric $\delta$ satisfies the relation
 \begin{equation*}
\begin{split}
\delta(A,B)&=
  \ln \max \left\{ \left\| L_{A,B} \right\|,
  \left\| L_{A,B}^{-1}\right\|\right\}.
\end{split}
\end{equation*}
\end{theorem}

In \cite{Po-automorphism},  we determined the group
$Aut([B(\cH)^n]_1)$  of all the free holomorphic automorphisms of
the noncommutative ball $[B(\cH)^n]_1$ and showed that if $\Psi\in
Aut([B(\cH)^n]_1)$ and $\lambda:=\Psi^{-1}(0)$, then there is a
unitary operator $U$ on $\CC^n$ such that
$$
\Psi=\Phi_U\circ \Psi_\lambda,
$$
where \begin{equation*}
  \Phi_U(X_1,\ldots
X_n):=[X_1\cdots  X_n]U , \qquad (X_1,\ldots, X_n)\in  [B(\cH)^n]_1,
\end{equation*}
and
\begin{equation*}
 \Psi_\lambda=- \Theta_\lambda(X_1,\ldots, X_n):={
\lambda}-\Delta_{ \lambda}\left(I_\cK-\sum_{i=1}^n \bar{{
\lambda}}_i X_i\right)^{-1} [X_1\cdots X_n] \Delta_{{\lambda}^*},
\end{equation*}
for some $\lambda=(\lambda_1,\ldots, \lambda_n)\in \BB_n$, where
$\Theta_\lambda$ is the characteristic function  of the row
contraction $\lambda$, and $\Delta_{ \lambda}$,
$\Delta_{{\lambda}^*}$ are certain defect operators.

 Let   $A:=(A_1,\ldots, A_n)$ and $B:=(B_1,\ldots, B_n)$ be in
$[B(\cH)^n]_1^-$.  We showed (see Lemma 2.3 from
\cite{Po-automorphism}) that
\begin{equation}
\label{equi2} A\overset{H}{{\underset{c}\prec}}\, B\quad \text{ if
and only if }\quad \Psi(A_1,\ldots,
A_n)\overset{H}{{\underset{c}\prec}}\, \Psi(B_1,\ldots B_n)
\end{equation}
for any $\Psi\in Aut([B(\cH)^n]_1)$. Using this result, we can prove
the following  proposition concerning the invariance of the norm of
$L_{B,A}$ under the automorphism group $Aut([B(\cH)^n]_1)$.

\begin{proposition}\label{inv}
Let   $A:=[A_1,\ldots, A_n]$ and $B:=[B_1,\ldots, B_n]$ be   in
$[B(\cH)^n]_1^-$ such that  $A\overset{H}{{\prec}}\, B$. If $\Psi\in
Aut([B(\cH)^n]_1)$, then
$$
\|L_{B,A}\|=\|L_{\Psi(B),\Psi(A)}\|.
$$
\end{proposition}

\begin{proof}
Let $c_0:=\|L_{B,A}\|$. According to Theorem \ref{equivalent2}, we
have  $ A\overset{H}{{\underset{c_0}\prec}}B$.  Consequently,
relation \eqref{equi2} implies
 $\Psi(A_1,\ldots,
A_n)\overset{H}{{\underset{c_0}\prec}}\, \Psi(B_1,\ldots B_n)$.
Using again Theorem \ref{equivalent2}, we deduce that
\begin{equation}
\label{L<L} \|L_{\Psi(B),\Psi(A)}\|\leq c_0=\|L_{B,A}\|.
\end{equation}
Since $\Psi\circ \Psi=id$, we  can apply  relation \eqref{L<L} when
$A$ and $B$ are replaced by $\Psi(A)$ and $\Psi(B)$, respectively,
and obtain $\|L_{B,A}\|\leq \|L_{\Psi(B),\Psi(A)}\|$. This completes
the proof.
\end{proof}

In what follows we calculate the norm of $L_{B,A}$ in terms of the
reconstruction operators.

\begin{theorem}\label{LCC}
If $A,B\in [B(\cH)^n]_1$, then
$$
\|L_{B,A}\|=\|C_AC_B^{-1}\|,
$$
where $C_X:=( \Delta_X\otimes I)(I-R_X)^{-1}$ and $R_X:=X_1^*\otimes
R_1+\cdots + X_n^*\otimes R_n$ is the reconstruction operator
associated with  the $n$-tuple $X:=(X_1,\ldots, X_n)\in
[B(\cH)^n]_1$. Moreover, if $A,B\in [B(\cH)^n]_1^-$ is such that
$A\overset{H}{\prec}\, B$, then
$$
\|L_{B,A}\|=\sup_{r\in[0,1)}\|C_{rA}C_{rB}^{-1}\|.
$$
\end{theorem}

\begin{proof} Since $A,B\in [B(\cH)^n]_1$,   Theorem 1.6 from \cite{Po-hyperbolic}, implies
$A\overset{H}{\sim}\, B$.  Let  $c\geq 1$ and  assume that
$P(rA,R)\leq c^2 P(rB,R)$ for any $r\in[0,1)$. Since $\|A\|<1$ and
$\|B\|<1$, we can take the limit, as $r\to 1$, in the operator norm
topology, and obtain
 $ P(A,R)\leq c^2 P(B,R)$. Conversely, if the latter inequality
 holds, then $P(A,S)\leq c^2 P(B,S)$,
where $S:=(S_1,\ldots, S_n)$ is  the $n$-tuple of left creation
operators.  Applying the noncommutative Poisson transform $\text{\rm
id}\otimes P_{rR }$, $r\in [0,1)$, and taking into account that it
is a positive map, we deduce that $  P(rA,R)\leq c^2 P(rB,R) $ for
any $r\in [0,1)$.

 Now, combining \eqref{equi} with Theorem
\ref{equivalent2},  we deduce that
\begin{equation}
\label{PPL}
 P(A,R)\leq c^2 P(B,R) \quad \text{ if and only if }
\quad \|L_{B,A}\|\leq c.
\end{equation}
 We
recall that the free pluriharmonic  kernel $P(X,R)$,
$X:=(X_1,\ldots, X_n)\in [B(\cH)^n]_1$, has the factorization $P(X,
R)=C_X^* C_X$, where $C_X:=(  \Delta_X\otimes I)(I-R_X)^{-1}$ and
$R_X:=X_1^*\otimes R_1+\cdots + X_n^*\otimes R_n$ is the
reconstruction operator. Note also that $C_X$ is an invertible
operator. Consequently,
$$P(A,R)\leq c^2 P(B,R)\quad \text{  if and only if }\quad
{C_B^*}^{-1} C_A^*C_AC_B^{-1}\leq c^2I.
$$
Setting $c_0:=\|C_AC_B^{-1}\|$, we have $P(A,R)\leq c_0^2 P(B,R)$.
Now, due to relation \eqref{PPL}, we deduce that
$$\|L_{B,A}\|\leq c_0=\|C_AC_B^{-1}\|.
$$
Setting  $c_0':=\|L_{B,A}\|$ and using again \eqref{PPL}, we obtain
$P(A,R)\leq {c_0'}^2 P(B,R)$. Hence, we deduce that ${C_B^*}^{-1}
C_A^*C_AC_B^{-1}\leq {c_0'}^2I$, which implies
$$\|C_AC_B^{-1}\|\leq
c_0'=\|L_{B,A}\|.
$$
Therefore, $\|L_{B,A}\|=\|C_AC_B^{-1}\|$. Using the first part of
this   theorem  and  Corollary \ref{Lr}, we complete the proof.
\end{proof}

We remark that, due to Theorem \ref{LCC} and Corollary
\ref{LBA-inf},   we have
\begin{equation}
\label{CACB}\|L_{B,A}\|=\|C_AC_B^{-1}\|
 =\inf\{c\geq 1:\  P( A, R)\leq c^2 P(B, R) \}
\end{equation}
for any $A,B\in [B(\cH)^n]_1$.

 Now, let us consider the
particular case when $n=1$.
 Due to the remarks above and  using
Proposition \ref{n=1} and Corollary \ref{LBA-inf}, we deduce  the
following. Let $ T$ and $ T'$ be in $[B(\cH)]_1^-$ such that
$T\overset{H}{\prec}\, T'$. Then
\begin{equation*}
\begin{split}
\|L_{T',T}\|&=\inf \{ c\geq 1: \ Q(rT,U)\leq c^2 Q(rT', U)\ \text{
for any } \ r\in
[0,1)\}\\
&=
 \inf\{c\geq 1: \ K(z,T)\leq c^2K(z,T')\ \text{ for
any } \ z\in \DD\}\\
&=\inf\{c\geq 1: \ K(z,T^*)\leq c^2K(z,{T'}^*)\ \text{  for
any } \ z\in \DD\}\\
&=\|L_{{T'}^*,T^*}\|
\end{split}
\end{equation*}
Therefore $T\overset{H}{\prec}\, T'$ if and only if
$T^*\overset{H}{\prec}\, {T'}^*$. In particular, if $T,T'\in
[B(\cH)]_1$, relation \eqref{CACB} implies
\begin{equation}\label{LTT'}
\begin{split}
\|L_{T',T}\|&=\|L_{{T'}^*, T^*}\|=\|C_{T^*} C_{{T'}^*}^{-1}\|\\
&= \sup_{e^{it}\in \TT}\|(I-T^*T)^{1/2} (I-e^{it}T)^{-1}(I-e^{it}
T')(I-{T'}^* T')^{-1/2}\|\\
&=\sup_{e^{it}\in \TT}\|(I-{T'}^* T')^{-1/2}(I-e^{it}
{T'}^*)(I-e^{it}T^*)^{-1}(I-T^*T)^{1/2}\|.
\end{split}
\end{equation}
Therefore, if $ T$ and $ T'$ are in $[B(\cH)]_1$ such that
$T\overset{H}{\prec}\, T'$, then
\begin{equation}
\label{Suciu} \|L_{T',T}\|=\sup_{e^{it}\in \TT}\|(I-{T'}^*
T')^{-1/2}(I-e^{it} {T'}^*)(I-e^{it}T^*)^{-1}(I-T^*T)^{1/2}\|,
\end{equation}
which is a result obtained by Suciu \cite{Su2} using different
methods.

Using Theorem \ref{LCC} and  relation \eqref{Suciu}, we can deduce
the following result, which is needed in the next section.

\begin{proposition}
\label{connection} If $A,B\in [B(\cH)^n]_1$, then
$$
\|L_{B,A}\|=\| L_{R_B, R_A}\|= \sup_{e^{it}\in \TT} \left\|(I-R_B
R_B^*)^{-1/2} (I-e^{it} R_B)(I-e^{it} R_A)^{-1} (I-R_A
R_A^*)^{1/2}\right\|
$$
where   $R_X:=X_1^*\otimes R_1+\cdots + X_n^*\otimes R_n$ is the
reconstruction operator.
\end{proposition}
 \begin{proof} Notice first that $R_A$ and $R_B$  are strict contractions.
 According to Theorem \ref{LCC}, using the
 noncommutative von Neumann inequality (\cite{Po-von}), and relation \eqref{LTT'}, we have
 \begin{equation*}
 \begin{split}
\|L_{B,A}\|&=\|\Delta_A\otimes I)(I-R_A)^{-1}
(I-R_B)(\Delta_B^{-1}\otimes I)\|\\
&=\sup_{e^{it}\in \TT} \left\|(I-R_A^* R_A)^{1/2}(I-e^{it} R_A)^{-1}
(I-e^{it} R_B) (I-R_B^* R_B)^{-1/2}\right\|\\
&=\left\|L_{R_B, R_A}\right\|=\left\|L_{R_B^*, R_A^*}\right\|.
 \end{split}
 \end{equation*}
 Using now  relation  \eqref{Suciu}, we can complete the proof.
\end{proof}

Taking into account  Theorem \ref{formula}, Proposition \ref{inv},
and Theorem \ref{LCC},  one can deduce the following result from
\cite{Po-hyperbolic}.

\begin{corollary} If   $A,B \in [B(\cH)^n]_1^-$  and  $A\overset{H}{\sim}\,
B$, then
$$\delta(\Psi(A), \Psi(B))=\delta(A,B)
$$
  for any
$\Psi\in Aut([B(\cH)^n]_1)$. Moreover, in this case we have
\begin{equation*}
\begin{split}
\delta(A,B)=\ln \max \left\{\sup_{r\in [0,1)}\left\|C_{rA}
C_{rB}^{-1} \right\|, \sup_{r\in [0,1)} \left\|C_{rB} C_{rA}^{-1}
\right\|\right\},
\end{split}
\end{equation*}
 where $C_X:=( \Delta_X\otimes I)(I-R_X)^{-1}$ and
$R_X:=X_1^*\otimes R_1+\cdots + X_n^*\otimes R_n$ is the
reconstruction operator.
\end{corollary}

\bigskip
\section{The geometric structure of the operator $L_{B,A}$}

This section  deals with  the geometric structure of the operator
$L_{B,A}$.  As consequences, we obtain new characterizations  for
the Harnack domination (resp.~equivalence) in $[B(\cH)^n]_1^-$.

 We need to recall from
\cite{Po-charact}--\cite{Po-analytic}
 a few facts
 concerning multi-analytic   operators on Fock spaces.
   We say that
 a bounded linear
  operator
$A$ acting from $F^2(H_n)\otimes \cK$ to $ F^2(H_n)\otimes \cG$ is
 multi-analytic
if
\begin{equation*}
A(S_i\otimes I_\cK)= (S_i\otimes I_{\cG}) A\quad \text{\rm for any
}\ i=1,\dots, n.
\end{equation*}
Note that $A$ is uniquely determined by the operator $\theta:\cK\to
F^2(H_n)\otimes \cG$, which is   defined by ~$\theta k:=A(1\otimes
k)$, \ $k\in \cK$,
 and
is called the  symbol  of  $A$.
We can associate with $A$ a unique formal Fourier expansion

\begin{equation*}
A\sim \sum_{\alpha \in \FF_n^+} R_\alpha \otimes \theta_{(\alpha)},
\end{equation*}
where $R_i$, \ $i=1,\ldots, n$, are the right creation operators on
$F^2(H_n)$, and  the coefficients
  $\theta_{(\alpha)}\in B(\cK, \cG)$,  are given by
 $$
\left< \theta_{(\tilde\alpha)}x,y\right>:=  \left< A(1\otimes x),
e_\alpha \otimes y\right>,\quad x\in \cK,\ y\in \cG,\ \alpha\in
\FF_n^+.
$$
Here $\tilde\alpha$ is the reverse of $\alpha$, i.e., $\tilde\alpha=
g_{i_k}\cdots g_{i_1}$ if $\alpha= g_{i_1}\cdots g_{i_k}$. Moreover,
in this case  we have
  $$A=\text{\rm SOT-}\lim_{r\to 1}\sum_{k=0}^\infty \sum_{|\alpha|=k}
   r^{|\alpha|} R_\alpha\otimes \theta_{(\alpha)},
   $$
   where, for each $r\in (0,1)$, the series converges in the uniform norm.
The set of  all multi-analytic operators in $B(F^2(H_n)\otimes \cK,
F^2(H_n)\otimes \cG)$  coincides  with $R_n^\infty\bar \otimes
B(\cK,\cG)$, the WOT-closed operator space  generated by the spatial
tensor product. A multi-analytic operator is called inner if it is
an isometry.

Let $T:=[T_1,\dots, T_n]$  be a row contraction, i.e., $
T_1T_1^*+\cdots +T_nT_n^*\leq I_\cH. $
We recall that the defect operators  associated with $T $ are given
by
\begin{equation*}
 \Delta_T:=\left( I_\cH-\sum_{i=1}^n
T_iT_i^*\right)^{1/2}\in B(\cH) \quad \text{ and }\quad
\Delta_{T^*}:= ([\delta_{ij}I_\cH-T_i^*T_j]_{n\times n})^{1/2}\in
B(\cH^{(n)}),
\end{equation*}
while the defect spaces are $\cD_T:=\overline{\Delta_T\cH}$ and
$\cD_{T^*}:=\overline{\Delta_{T^*}\cH^{(n)}}$, where $\cH^{(n)}$
denotes the direct sum of $n$ copies of $\cH$. For each $i=1,\ldots,
n$, let $D_{T^*,i}:\cH\to \CC\otimes \cD_{T^*}\subset
F^2(H_n)\otimes \cD_{T^*}$ be defined by
$$
D_{T^*,i}h:= 1\otimes
\Delta_{T^*}(\underbrace{0,\ldots,0}_{{i-1}\mbox{ \scriptsize
times}},h,0,\ldots).
$$
Consider the Hilbert space $\cK_T:=\cH\oplus [F^2(H_n)\otimes
\cD_{T^*}]$ and define $V_i:\cK_T\to\cK_T$, $i=1,\ldots, n$,  by
\begin{equation}\label{dil}
V_i(h\oplus \xi):= T_ih \oplus [D_{T^*,i}h +(S_i\otimes
I_{\cD_{T^*}})\xi]
\end{equation}
for any $h\in \cH$ and  $\xi\in F^2(H_n)\otimes\cD_{T^*}$. Notice
that
\begin{equation}\label{dilmatr}
V_i=\left[\begin{matrix} T_i& 0\\
D_{T^*,i}& S_i\otimes I_{\cD_{T^*}}
\end{matrix}\right]
\end{equation}
with respect to the decomposition $\cK_T=\cH\oplus [F^2(H_n)\otimes
\cD_{T^*}]$. It was proved in \cite{Po-isometric} that the $n$-tuple
 $V:=[V_1,\dots, V_n]$  is the   minimal isometric dilation of $T$.

The main result of this section is the following multivariable
generalization of Suciu's result from \cite{Su2}.

\begin{theorem}\label{geometric}
 Let   $A:=[A_1,\ldots, A_n]$ and $B:=[B_1,\ldots,
B_n]$  be    in $[B(\cH)^n]_1^-$ and let $V:=[V_1,\ldots, V_n]$ on
$\cK_A\supseteq \cH$ and $W:=[W_1,\ldots, W_n]$ on $\cK_B\supseteq
\cH$ be the minimal isometric dilations of $A$ and $B$,
respectively. Then
 $A\overset{H}{{\prec}}\, B$ if and only if  there exit two bounded linear operators
 $\Omega_0:\cH\to \cD_{B^*}$ and $\Theta_0:\cD_{A^*}\to \cD_{B^*}$ such that
 the following conditions
 are satisfied:

 \begin{enumerate}
 \item[(i)] $\Delta_{B^*} \Omega_0
 =\left[\begin{matrix} A_1^*-B_1^*\\ \vdots\\ A_n^*-B_n^*\end{matrix}\right];$

 \item[(ii)] the map $\Omega:\cH\to F^2(H_n)\otimes \cD_{B^*}$
 defined by
 \begin{equation*}
 \begin{split}
 \Omega h&:=(I_{F^2(H_n)}\otimes\Omega_0)
 \sum_{k=0}^\infty\left(\sum_{j=1}^n R_j\otimes A_j^*\right)^k(1\otimes h)
 =\sum_{\alpha\in \FF_n^+} e_\alpha \otimes \Omega_0 A_\alpha^*h
 \end{split}
 \end{equation*}
 is a bounded operator;

 \item[(iii)] $\Delta_{A^*}=\Delta_{B^*} \Theta_0$;

 \item[(iv)] the formal series
 $$
 I_{F^2(H_n)}\otimes \Theta_0+ (I_{F^2(H_n)}\otimes
 \Omega_0)\left(I-\sum_{j=1}^n R_j\otimes
 A_j^*\right)^{-1}[R_1\otimes I_\cH,\ldots, R_n\otimes
 I_\cH](I_{F^2(H_n)}\otimes \Delta_{A^*})
 $$
is the  Fourier representation of a multi-analytic operator
$\Theta:F^2(H_n)\otimes \cD_{A^*}\to F^2(H_n)\otimes \cD_{B^*}$.
 \end{enumerate}

In this case the operator $L_{B,A}:\cK_B\to \cK_A$ satisfying
$L_{B,A}|_\cH=I_\cH$ and $L_{B,A}W_i=V_i L_{B,A}$, $i=1,\ldots, n$,
is given by
\begin{equation}
\label{LBA*}
L_{B,A}^*=\left[\begin{matrix} I_\cH& 0\\
\Omega &  \Theta
\end{matrix}\right].
\end{equation}
\end{theorem}

\begin{proof}Assume that $A\overset{H}{{\prec}}\, B$.  According to
Theorem \ref{equivalent2}, there is a unique  bounded operator
$L_{B,A}:\cK_B\to \cK_A$ satisfying $L_{B,A}|_\cH=I_\cH$ and
$L_{B,A}W_i=V_i L_{B,A}$ for   $i=1,\ldots, n$. Consequently, due to
relation \eqref{dilmatr}, the operator $L_{B,A}^*:\cH\oplus
[F^2(H_n)\otimes \cD_{A^*}] \to \cH\oplus [F^2(H_n)\otimes
\cD_{B^*}]$ has the operator matrix  representation
$$
L_{B,A}^*=\left[\begin{matrix} I_\cH& 0\\
\Omega &  \Theta
\end{matrix}\right],
$$
where $\Omega:\cH\to  F^2(H_n)\otimes \cD_{B^*}$ and
$\Theta:F^2(H_n)\otimes \cD_{A^*}\to F^2(H_n)\otimes \cD_{B^*}$ are
bounded operators. Moreover, we have

$$
\left[\begin{matrix} I_\cH& 0\\
\Omega &  \Theta
\end{matrix}\right]\left[\begin{matrix} A_i^*& D_{A^*,i}^*\\
0 &  S_i\otimes I_{\cD_{A^*}}
\end{matrix}\right]=\left[\begin{matrix} B_i^*& D_{B^*,i}^*\\
0 &  S_i\otimes I_{\cD_{B^*}}
\end{matrix}\right]\left[\begin{matrix} I_\cH& 0\\
\Omega &  \Theta
\end{matrix}\right].
$$
Hence,  for each $i=1,\ldots, n$, we deduce the following relations:

\begin{equation}
\label{AB} A_i^*=B_i^*+D_{B^*,i}^* \Omega,
\end{equation}
\begin{equation}
\label{Om} \Omega A_i^*=  ( S_i^*\otimes I_{\cD_{B^*}}) \Omega,
\end{equation}
\begin{equation}
\label{Da} D_{A^*,i}^*= D_{B^*,i}^* \Theta,
\end{equation}
 and
\begin{equation}
\label{OD} \Omega D_{A^*,i}^*= ( S_i^*\otimes I_{\cD_{B^*}})\Theta-
\Theta ( S_i^*\otimes I_{\cD_{A^*}}).
\end{equation}

Define $\Omega_0:\cH\to \cD_{B^*}$ by setting  $\Omega_0:=J_{B^*}
(P_\CC\otimes I_{\cD_{B^*}}) \Omega$, where $J_{B^*}:\CC\otimes
\cD_{B^*}\to \cD_{B^*}$ is the unitary operator defined by
$J_{B^*}(1\otimes d)=d$ for $d\in \cD_{B^*}$. Let $P_i:\cH^{(n)}\to
\cH$ be the orthogonal projection on the $i$-component of the direct
sum $\cH^{(n)}$. A straightforward computation shows that
\begin{equation}
\label{PiD} D_{B^*,i}^* f=P_i\Delta_{B^*} J_{B^*}(P_\CC\otimes
I_{\cD_{B^*}})f, \qquad f\in F^2(H_n)\otimes \cD_{B^*}.
\end{equation}
A similar relation holds if $B$ is replaced by $A$. Due to relations
\eqref{PiD} and  \eqref{AB}, we deduce that
\begin{equation*}
P_i\Delta_{B^*}\Omega_0=P_i\Delta_{B^*}J_{B^*}(P_\CC\otimes
I_{\cD_{B^*}})\Omega=D_{B^*,i}^* \Omega=A_i^*-B_i^*,
\end{equation*}
which implies (i). Now, since $\Omega:\cH\to  F^2(H_n)\otimes
\cD_{B^*}$ is a bounded linear operator,  we have \begin{equation}
\label{Ome} \Omega h=\sum_{\alpha\in \FF_n}e_\alpha \otimes
\Omega_{(\alpha)}h, \quad h\in \cH,
\end{equation}
where $\Omega_{(\alpha)}:\cH\to \cD_{B^*}$, $\alpha\in \FF_n^+$, are
bounded operators with $\sum_{\alpha\in\FF_n^+}
\|\Omega_{(\alpha)}h\|^2\leq C\|h\|^2$, $h\in \cH$, for some
constant $C>0$. Since $I_{F^2(H_n)}-\sum_{i=1}^n S_iS_i^*=P_\CC$,
the orthogonal projection on the constants, and using relations
\eqref{Om} and \eqref{Ome}, we deduce that
\begin{equation*}
\begin{split}
\Omega h&= (I_{F^2(H_n)}\otimes \Omega_0)(1\otimes
h)+\left(\sum_{i=1}^n S_iS_i^*\otimes I_{\cD_{B^*}}\right) \Omega
h\\
&=(I_{F^2(H_n)}\otimes \Omega_0)(1\otimes h)+ \sum_{i=1}^n
(S_i\otimes I_{\cD_{B^*}})\Omega A_i^*h\\
&= (I_{F^2(H_n)}\otimes \Omega_0)(1\otimes h)+ \sum_{i=1}^n
(S_i\otimes I_{\cD_{B^*}}) \left( \sum_{\beta\in \FF_n}e_\beta
\otimes \Omega_{(\beta)}A_i^*h\right).
\end{split}
\end{equation*}
Hence, we obtain
$$
\sum_{\alpha\in \FF_n}e_\alpha \otimes \Omega_{(\alpha)}h =
(I_{F^2(H_n)}\otimes \Omega_0)(1\otimes h) +\sum_{\beta\in
\FF_n}(I\otimes \Omega_{(\beta)})\left(\sum_{i=1}^n S_i\otimes
A_i^*\right)(e_\beta\otimes h)
$$
for any $h\in \cH$. Therefore,  for each  $k=1,2,\ldots$, we have
\begin{equation}
\label{formu} \sum_{|\alpha|=k}e_\alpha \otimes \Omega_{(\alpha)}h =
\sum_{|\beta|=k-1}(I\otimes \Omega_{(\beta)})\left(\sum_{i=1}^n
S_i\otimes A_i^*\right)(e_\beta\otimes h).
\end{equation}
In particular, if $k=1$, we  deduce that
\begin{equation}
\label{prima} \sum_{j=1}^n e_j\otimes
\Omega_{(g_j)}h=(I_{F^2(H_n)}\otimes \Omega_0)\left(\sum_{j=1}^n
e_j\otimes A_j^*h\right).
\end{equation}
Now, we prove by induction that
\begin{equation}
\label{m} \sum_{|\alpha|=m}e_\alpha \otimes \Omega_{(\alpha)}h =
  (I_{F^2(H_n)}\otimes
\Omega_0) \left(\sum_{j=1}^n R_j\otimes A_j^*\right)^m(1\otimes h)
\end{equation}
for any $m=1,2,\ldots$. Assume that \eqref{m} holds for $m=k$. Then,
due to \eqref{formu},  we have
\begin{equation*}
\begin{split}
 \sum_{|\alpha|=k+1}e_\alpha \otimes \Omega_{(\alpha)}h
 &=\sum_{|\beta|=k}(I\otimes \Omega_{(\beta)})\left(\sum_{i=1}^n
S_i\otimes A_i^*\right)(e_\beta\otimes h)\\
&=\sum_{i=1}^n (S_i\otimes I)\left(\sum_{|\beta|=k} e_\beta \otimes
\Omega_{(\beta)} A_i^*h\right)\\
&=\sum_{i=1}^n (S_i\otimes I)(I_{F^2(H_n)}\otimes \Omega_0)
\left(\sum_{j=1}^n R_j\otimes A_j^*\right)^k(1\otimes A_i^*h)\\
&=(I_{F^2(H_n)}\otimes \Omega_0) \left(\sum_{j=1}^n R_j\otimes
A_j^*\right)^k \sum_{i=1}^n(S_i\otimes I)(1\otimes A_i^*h)\\
&=(I_{F^2(H_n)}\otimes \Omega_0) \left(\sum_{j=1}^n R_j\otimes
A_j^*\right)^{k+1}(1\otimes h).
\end{split}
\end{equation*}
Here we also used the fact  that $S_i R_j=R_jS_i$ for any
$i,j=1,\ldots, n$. Therefore, relation \eqref{m} holds.
Consequently, part (ii) follows.

To prove part (iii), let $\Theta_0:\cD_{A^*}\to \cD_{B^*}$ be
defined by
$$
\Theta_0 x:= J_{B^*}(P_\CC\otimes I_{\cD_{B^*}})\Theta (1\otimes
x),\quad x\in \cD_{A^*}.
$$
Due to relations  \eqref{PiD} and \eqref{Da}, we have

\begin{equation} \label{PiDOD}
 P_i\Delta_{A^*} J_{A^*}(P_\CC\otimes I_{\cD_{A^*}})f=D_{A^*,i}^* f
 =D_{B^*,i}^* \Theta f=P_i\Delta_{B^*} J_{B^*}(P_\CC\otimes
 I_{\cD_{B^*}})(\Theta f)
\end{equation}
for any $ f\in F^2(H_n)\otimes \cD_{A^*}$ and $i=1,\ldots, n$. In
the particular case when $f=1\otimes x$, $x\in \cD_{A^*}$, relation
\eqref{PiDOD} implies $P_i\Delta_{A^*}x=P_i\Delta_{B^*} \Theta_0x$,
which proves part (iii).

On the other hand, if $f=e_\gamma\otimes x$ with $\gamma\in
\FF_n^+$, $|\gamma|\geq 1$, and $x\in \cD_{A^*}$, then relation
\eqref{PiDOD} implies $\Delta_{B^*} J_{B^*}(P_\CC\otimes
 I_{\cD_{B^*}})(\Theta(e_\gamma\otimes x))=0$. Since $\Delta_{B^*}$
 is one-to-one on $\cD_{B^*}$, we deduce that
 \begin{equation}
 \label{eq=0}
(P_\CC\otimes
 I_{\cD_{B^*}})(\Theta(e_\gamma\otimes x))=0.
 \end{equation}
Now, due to relations \eqref{OD} and \eqref{PiD}, we have
\begin{equation}
\label{OD2} \Omega P_i\Delta_{A^*} J_{A^*}(P_\CC\otimes
I_{\cD_{A^*}})f=
 \Omega D_{A^*,i}^*f= ( S_i^*\otimes I_{\cD_{B^*}})\Theta f-
\Theta ( S_i^*\otimes I_{\cD_{A^*}})f
\end{equation}
for any $f\in F^2(H_n)\otimes \cD_{A^*}$. In the particular case
when $f=1\otimes x$, $x\in \cD_{A^*}$, relation \eqref{OD2} implies
\begin{equation}
\label{DPD} \Omega P_i\Delta_{A^*} x=( S_i^*\otimes
I_{\cD_{B^*}})\Theta(1\otimes x).
\end{equation}
Using  relation \eqref{DPD}, part (ii) of this theorem,  and the
fact that $I_{F^2(H_n)}-\sum_{i=1}^n S_iS_i^*=P_\CC$, we deduce that
\begin{equation*}
\begin{split}
\Theta(1\otimes x)&= (I\otimes \Theta_0)(1\otimes
x)+\left(\sum_{i=1}^n
S_iS_i^*\otimes I_{\cD_{B^*}}\right))\Theta(1\otimes x)\\
&= (I\otimes \Theta_0)(1\otimes x)+\sum_{i=1}^n \left(S_i\otimes
I_{\cD_{B^*}}\right)\Omega P_i\Delta_{A^*}x\\
&= (I\otimes \Theta_0)(1\otimes x)+\sum_{i=1}^n \left(S_i\otimes
I_{\cD_{B^*}}\right)\left(\sum_{\alpha\in \FF_n^+} e_\alpha \otimes
\Omega_0 A_\alpha^*P_i\Delta_{A^*} x\right)\\
&= (I\otimes \Theta_0)(1\otimes x) +(I\otimes \Omega_0)
\sum_{k=0}^\infty \left(\sum_{|\alpha|=k} R_{\tilde \alpha}\otimes
A_\alpha^*\right) \sum_{i=1}^n  \left(R_i\otimes
I_\cH\right)(1\otimes P_i\Delta_{A^*} x)\\
&= (I\otimes \Theta_0)(1\otimes x) +(I\otimes \Omega_0)
\sum_{k=0}^\infty\left(\sum_{j=1}^n R_j\otimes A_j^*\right)^k
[R_1\otimes I_\cH,\ldots, R_n\otimes
 I_\cH](I \otimes \Delta_{A^*})(1\otimes x)
\end{split}
\end{equation*}
for any $x\in \cD_{A^*}$.

Now we prove that $\Theta$ is a multi-analytic operator, i.e.,
$\Theta ( S_i\otimes I_{\cD_{A^*}})=(S_i\otimes I_{\cD_{B^*}})
\Theta$ for any $i=1,\ldots, n$. Using relation \eqref{OD2} when
$f=e_{g_i \beta}\otimes x$ with $\beta\in \FF_n^+$ and $x\in
\cD_{A^*}$, we obtain
$$
( S_i^*\otimes I_{\cD_{B^*}})\Theta (e_{g_i \beta}\otimes x)- \Theta
( S_i^*\otimes I_{\cD_{A^*}})(e_{g_i \beta}\otimes x)=\Omega P_i
\Delta_{A^*}(0)=0.
$$
Hence, we have
\begin{equation} \label{anal1}
\Theta(e_\beta\otimes x)=( S_i^*\otimes I_{\cD_{A^*}})\Theta (e_{g_i
\beta}\otimes x)
\end{equation}
for any $\beta\in \FF_n^+$ and $i=1,\ldots, n$. Similarly, if $i\neq
j$, then \eqref{OD2} implies
\begin{equation}
\label{anal2}
 (S_j^*\otimes I_{\cD_{B^*}})\Theta(e_{g_i\beta}\otimes
x)=0.
\end{equation}
Now, using relations \eqref{anal1}, \eqref{anal2}, and \eqref{eq=0},
we obtain
\begin{equation*}
\begin{split}
\Theta(e_{g_i\beta}\otimes x) &= (P_\CC\otimes
I_{\cD_{B^*}})\Theta(e_{g_i\beta}\otimes x) +\left( \sum_{j=1}^n
S_jS_j^*\otimes
I_{\cD_{B^*}}\right)\Theta(e_{g_i\beta}\otimes x)\\
&=(S_i\otimes I_{\cD_{B^*}})(S_i^*\otimes
I_{\cD_{B^*}})\Theta(e_{g_i\beta}\otimes x)\\
&=(S_i\otimes I_{\cD_{B^*}}) \Theta (e_\beta\otimes x).
\end{split}
\end{equation*}
Consequently, we deduce that
$$
\Theta ( S_i\otimes I_{\cD_{A^*}})(e_\beta\otimes x)=(S_i\otimes
I_{\cD_{B^*}}) \Theta(e_\beta\otimes x)
$$
for any $\beta\in \FF_n^+$ and $x\in \cD_{A^*}$, which shows that
$\Theta$ is a multi-analytic operator. Hence, using the calculations
for $\Theta(1\otimes x)$, and the fact that $S_iR_j=R_jS_i$ for
$i,j=1,\ldots, n$, we deduce that
\begin{equation*}
\Theta(e_\alpha\otimes x)= (I\otimes \Theta_0)(e_\alpha\otimes x)
+(I\otimes \Omega_0) \sum_{k=0}^\infty\left(\sum_{j=1}^n R_j\otimes
A_j^*\right)^k [R_1\otimes I_\cH,\ldots, R_n\otimes
 I_\cH](I \otimes \Delta_{A^*})(e_\alpha\otimes x)
\end{equation*}
for any $\alpha\in \FF_n^+$ and $x\in \cD_{A^*}$. Therefore,
$\Theta$ is a multi-analytic operator with Fourier representation
given in part (iv).

Now, we prove the converse of this theorem. Assume that there exit
two bounded linear operators
 $\Omega_0:\cH\to \cD_{B^*}$ and $\Theta_0:\cD_{A^*}\to \cD_{B^*}$ such that
 the  conditions (i)-(iv)
 are satisfied.
 Let $X:\cK_B\to \cK_A$  be defined by
$$
X^*:=\left[\begin{matrix} I_\cH& 0\\
\Omega &  \Theta
\end{matrix}\right].
$$
Since $X|\cH=I_\cH$, due to Theorem \ref{equivalent2}, it remains to
show that $XW_i=V_i X$ for any  $i=1,\ldots, n$. As we saw at the
beginning of this proof, it is enough to show that  that relations
\eqref{AB}, \eqref{Om}, \eqref{Da}, and \eqref{OD} hold.

First, notice that condition  (i) implies  $P_i\Delta_{B^*}
\Omega_0=A_i^*-B_i^*$, $i=1,\ldots,n$. Using relation  \eqref{PiD}
and part (ii), we deduce \eqref{AB}. Similarly, conditions  (iii)
and (iv) imply \eqref{Da}. To prove \eqref{Om}, note that, using
condition  (ii) and the fact that $S_i^* S_j=\delta_{ij}I$ for
$i,j=1,\ldots, n$, we have
\begin{equation*}
\begin{split}
( S_i^*\otimes I_{\cD_{B^*}}) \Omega&= ( S_i^*\otimes
I_{\cD_{B^*}})(I\otimes \Omega_0)\left(\sum_{k=0}^\infty
\sum_{|\alpha|=k}
(S_\alpha \otimes A_\alpha^*) (1\otimes h)\right)\\
&=(I\otimes \Omega_0) \left(\sum_{k=0}^\infty \sum_{|\beta|=k}
(S_\beta \otimes A_\beta^*) (1\otimes A_i^*h)\right)\\
&=\Omega A_i^*h
\end{split}
\end{equation*}
for any $h\in \cH$ and $i=1,\ldots, n$. It remains to prove relation
\eqref{OD}.

Let $f=e_{g_i\beta}\otimes x$ with $\beta\in \FF_n^+$,
$i,j=1,\ldots, n$,  and  $x\in \cD_{A^*}$. Note that, since $\Theta$
is multi-analytic operator and $S_j^* S_i=\delta_{ij} I$, we have
\begin{equation*}
\begin{split}
(S_j^*\otimes I_{\cD_{B^*}})\Theta f&- \Theta ( S_j^*\otimes
I_{\cD_{A^*}})f \\
&= (S_j^*\otimes I_{\cD_{B^*}})\Theta (S_iS_\beta \otimes
I_{\cD_{A^*}})(1\otimes x)- \Theta ( S_j^*\otimes
I_{\cD_{A^*}})(S_iS_\beta \otimes I_{\cD_{A^*}})(1\otimes x)\\
&= (\delta_{ij} I\otimes I_{\cD_{B^*}}) \Theta (S_\beta\otimes
I_{\cD_{A^*}})(1\otimes x)-\Theta (\delta_{ij} I\otimes
I_{\cD_{A^*}})(e_\beta\otimes x)\\
&=0
\end{split}
\end{equation*}
and \ $
 \Omega D_{A^*,i}^* f=\Omega P_j\Delta(0)=0$. Therefore, in this
 case, we have
 $$
 (S_j^*\otimes I_{\cD_{B^*}})\Theta f- \Theta ( S_j^*\otimes
I_{\cD_{A^*}})f=\Omega D_{A^*,i}^* f=0.
$$

Note also that part (iv) and part (ii) imply
\begin{equation*}
\begin{split}
(S_j^*\otimes I_{\cD_{B^*}})&\Theta (1\otimes x)- \Theta (
S_j^*\otimes
I_{\cD_{A^*}})(1\otimes x) \\
&= (S_j^*\otimes I_{\cD_{B^*}})(I\otimes \Omega_0) \sum_{k=0}^\infty
\left[\left(\sum_{p=1}^n R_p\otimes A_p^*\right)^k [R_1\otimes
I_\cH,\ldots, R_n\otimes
 I_\cH](I \otimes \Delta_{A^*})(1\otimes x)\right]\\
 &=
 ( S_j^*\otimes \Omega_0)
\sum_{k=0}^\infty \left[\left(\sum_{p=1}^n R_p\otimes A_p^*\right)^k
[S_1\otimes I_\cH,\ldots, S_n\otimes
 I_\cH](1 \otimes \Delta_{A^*}x)\right]\\
 &=
 (S_j^*\otimes  \Omega_0)
\sum_{k=0}^\infty \left[\left(\sum_{p=1}^n R_p\otimes A_p^*\right)^k
\sum_{i=1}^n (S_i\otimes I_\cH)(1\otimes  P_i\Delta_{A^*}x)\right]\\
 &=
 \sum_{i=1}^n (S_j^*S_i\otimes  \Omega_0)\sum_{k=0}^\infty
 \left[\left(\sum_{p=1}^n R_p\otimes A_p^*\right)^k
 (1\otimes  P_i\Delta_{A^*}x)\right]\\
  &=
 \Omega (P_j\Delta_{A^*}x)=\Omega D^*_{A^*,j}(1\otimes x)
\end{split}
\end{equation*}
for any $x\in \cD_{A^*}$ and $j=1,\ldots, n$. Therefore, relation
\eqref{OD} holds. The proof is complete.
\end{proof}

If $L_{B,A}^*$ is given by  relation \eqref{LBA*},  we set
$\Omega_{B,A}:=\Omega$ and $\Theta_{B,A}:=\Theta$. Using Theorem
\ref{geometric}, one can easily obtain the following result.

\begin{corollary}
\label{co1} Let $A,B,C\in [B(\cH)^n]_1^-$.
\begin{enumerate}
\item[(i)] If $A\overset{H}{{\prec}}\, B$  and $B\overset{H}{{\prec}}\,
C$, then $A\overset{H}{{\prec}}\, C$. In this case, we have
$$L_{C,A}^*=L_{C,B}^* L_{B,A}^* \quad \text{ and }\quad
\Theta_{C,A}=\Theta_{C,B} \Theta_{B,A}.
$$
\item[(ii)] If $A\overset{H}{{\prec}}\, B$  then $A\overset{H}{\sim}\, B$ if and
only if $\Theta_{B,A}$ is  an invertible multi-analytic operator. In
this case, we have $$\Theta_{B,A}^{-1}=\Theta_{A,B}.$$
\end{enumerate}
\end{corollary}

Consider now the particular case when $A, B\in [B(\cH)^n]_1^-$  and
$B=0$. Note that Theorem \ref{geometric} implies  $ \Omega_0
 =\left[\begin{matrix} A_1^*\\ \vdots\\ A_n^*\end{matrix}\right]$
 and $\Theta_0=\Delta_{A^*}$.
 In this case, one can use Theorem \ref{geometric} and  some results
 from \cite{Po-hyperbolic} to obtain the following consequences.
We recall that the  joint spectral radius associated
 with an  $n$-tuple of operators
$ X:=(X_1,\ldots, X_n)\in B(\cH)^n$  is given by
$$
r(X):=\lim_{k\to \infty}\left\|\sum_{|\alpha|=k} X_\alpha
X_\alpha^*\right\|^{1/2k}.
$$
Moreover, $r(X)$ is equal to the spectral radius of the
reconstruction operator $ \sum_{i=1}^n X_i^*\otimes R_i$.

\begin{corollary} If $A\in [B(\cH)^n]_1^-$, then the following
statements are equivalent:
\begin{enumerate}
\item[(i)] $A\overset{H}{{\prec}}\, 0$;
\item[(ii)]
$\Omega_{0,A}$ and $\Theta_{0,A}$ are bounded operators;
\item[(iii)]
the joint spectral radius $r(A)<1$.
\end{enumerate}
\end{corollary}

We remak that when $n=1$, the equivalence $(i)\leftrightarrow (iii)$
was obtained by Ando-Suciu-Timotin in \cite{AST} and the equivalence
$(i)\leftrightarrow (ii)$ was obtained by Suciu in \cite{Su2}.

\begin{corollary} If $A\in [B(\cH)^n]_1^-$, then the following
statements are equivalent:
\begin{enumerate}
\item[(i)] $A\overset{H}{\sim}\, 0$ ;
\item[(ii)]
$\Omega_{0,A}$  is  bounded and  $\Theta_{0,A}$ is invertible;
\item[(iii)]
$A\in [B(\cH)^n]_1$.
\end{enumerate}
\end{corollary}

We remak that when $n=1$, the equivalence $(i)\leftrightarrow (iii)$
was obtained by Foia\c s in \cite{Fo} and the equivalence
$(i)\leftrightarrow (ii)$ was obtained by Suciu in \cite{Su2}.

\bigskip

\section{Schwarz-Pick lemma with respect to the hyperbolic distance}

In this section, we obtain  a Schwartz-Pick  lemma for contractive
free holomorphic  functions $F$ on $[B(\cH)^n]_1$ with respect to
the intertwining operator $L_{B,A}$. As a consequence, we deduce
that
$$
\delta(F(z), F(\xi))\leq \delta(z,\xi), \qquad z,\xi\in \BB_n,
$$
where $\delta$ is the hyperbolic distance.

We recall that a free holomorphic function $G:[B(\cH)^n]_1\to B(\cH)
\otimes_{min} B(\cE)$ is bounded if
$$
\|G\|_\infty:=\sup  \|G(X_1,\ldots, X_n)\|<\infty,
$$
where the supremum is taken over all $n$-tuples  of operators
$[X_1,\ldots, X_n]\in [B(\cK)^n]_1$ and  $\cK$ is an infinite
dimensional Hilbert space. We know that (see \cite{Po-holomorphic},
\cite{Po-pluriharmonic})  a bounded free holomorphic function $G$ is
uniquely determined by  its boundary function $\widetilde
G(R_1,\ldots, R_n)\in R_n^\infty\bar \otimes B(\cE)$ defined by
$$\widetilde G(R_1,\ldots, R_n)=\text{\rm SOT-}\lim_{r\to 1}
G(rR_1,\ldots, rR_n). $$ Moreover, $G$ is essentially the
noncommutative Poisson transform of $\widetilde G(R_1,\ldots, R_n)$.

 First, we need the
following result.

\begin{theorem} \label{S-P} Let $F_j:[B(\cH)^n]_1\to B(\cH) \otimes_{min}
B(\cE)$, $j=1,\ldots, m$, be free holomorphic functions with
coefficients in $B(\cE)$, and assume that $F:=(F_1,\ldots, F_m)$ is
a contractive free holomorphic function. If $W\in [B(\cH)^n]_1$,
then $F(0) \overset{H}{{\prec}}\, F(W)$
  and
$$
\left\|L_{ F(W), F(0)}\right\|\leq
\left(\frac{1+\|W\|}{1-\|W\|}\right)^{1/2}.
$$
\end{theorem}

\begin{proof}
First, we prove the theorem when $F$ is a strictly contractive free
holomorphic function. According to Lemma 1.2 from
\cite{Po-automorphism}, we must have $\|\widetilde F(rR_1,\ldots,
rR_n)\|<1$ for any $r\in [0,1)$, where $\widetilde F(R_1,\ldots,
R_n)$ is the boundary function of $F$. Denote by ${\bf R}_1,\ldots,
{\bf R}_m$ the right creation operators on the full Fock space with
$m$ generators, $F^2(H_m)$. Define the free holomorphic function
$\Psi:[B(\cH)^n]_1\to B(\cH)\otimes_{min} B(\cE\otimes  F^2(H_m))$
by
\begin{equation}
\label{defi} \Psi(X_1,\ldots, X_n):=\sum_{j=1}^m F_j(X_1,\ldots,
X_n)\otimes {\bf R}_j^*,\qquad (X_1,\ldots, X_n)\in [B(\cH)^n]_1.
\end{equation}
 Since ${\bf R}_j^* {\bf R}_p=\delta_{pj}I$, $p,j=1,\ldots, m$,   the boundary
function $\widetilde \Psi(R_1,\ldots, R_n)\in B(F^2(H_n)\otimes
\cE\otimes F^2(H_m))$ has the property that
$$
\|\widetilde \Psi(rR_1,\ldots, rR_n)\| =\left\|\sum_{j=1}\widetilde
F_j(rR_1,\ldots, rR_n)\widetilde F_j(rR_1,\ldots, rR_n)^*
\right\|^{1/2} =\|\widetilde F(rR_1,\ldots, rR_n)\|<1
$$
for any $r\in [0,1)$. Therefore, $\Psi$ is a strictly contractive
free holomorphic function. In particular, we have $\|\widetilde
\Psi(0)\|<1$.

Let $\Gamma:[B(\cH)^n]_1\to B(\cH)\otimes_{min} B(\cE\otimes
F^2(H_m))$ be the bounded free holomorphic function having the
boundary function
\begin{equation}
\label{Gamm} \widetilde \Gamma(R_1,\ldots, R_n):=-\widetilde
{\Psi}(0)+D_{\widetilde {\Psi}(0)^*}\widetilde {\Psi}(R_1,\ldots,
R_n)\left[I-\widetilde {\Psi}(0)^*\widetilde {\Psi}(R_1,\ldots,
R_n)\right]^{-1} D_{\widetilde \Psi(0)},
\end{equation}
where $D_Y:=(I-Y^*Y)^{1/2}$.  Note that $\Gamma$ is a  bounded free
holomorphic function with $\Gamma(0)=0$. Moreover, one can prove
that $\Gamma$ is contractive. Indeed, for each $y\in F^2(H_n)\otimes
\cE\otimes F^2(H_m)$ set $$\omega:=\left[I-\widetilde
{\Psi}(0)^*\widetilde {\Psi}(R_1,\ldots, R_n)\right]^{-1}
D_{\widetilde \Psi(0)}y.
$$
Since $ \left[\begin{matrix}
-\widetilde \Psi(0) & D_{\widetilde \Psi(0)^*}\\
D_{\widetilde \Psi(0)}& \widetilde \Psi(0)^*\end{matrix}\right] $ is
a unitary operator and
$$
\left[\begin{matrix}\widetilde \Gamma(R_1,\ldots, R_n)y\\
 D_{\widetilde \Psi(0)} y +\widetilde {\Psi}(0)^*\widetilde {\Psi}(R_1,\ldots,
R_n) \omega
 \end{matrix}\right]
  =
\left[\begin{matrix}
-\widetilde \Psi(0) & D_{\widetilde \Psi(0)^*}\\
D_{\widetilde \Psi(0)}& \widetilde \Psi(0)^*\end{matrix}\right]
\left[\begin{matrix} y\\
  \widetilde {\Psi}(R_1,\ldots,
R_n) \omega,
 \end{matrix}\right]
$$
  we deduce that

\begin{equation*}
\begin{split}
\|\widetilde \Gamma(R_1,\ldots, R_n)y\|&= \|y\|^2 +
 \left\| \widetilde {\Psi}(R_1,\ldots, R_n)\omega\right\|^2
 -\left\|D_{\widetilde \Psi(0)} y +\widetilde
{\Psi}(0)^*\widetilde {\Psi}(R_1,\ldots, R_n)\omega\right\|^2\\
&= \|y\|^2-\left(\left\|\omega\right\|^2   -\left\|\widetilde
{\Psi}(R_1,\ldots, R_n) \omega\right\|^2\right).
\end{split}
\end{equation*}
Now, since $\widetilde \Psi(R_1,\ldots, R_n)$ is a contraction, we
deduce that $\widetilde \Gamma(R_1,\ldots, R_n)$ is a contraction.
Therefore, $\Gamma$ is a contractive free holomorphic function with
$\Gamma(0)=0$. Due to  a noncommutative version of Gleason problem
(see Theorem 2.4 from \cite{Po-holomorphic}), we have
\begin{equation}
\label{gafa} \widetilde \Gamma(R_1,\ldots, R_n)=\sum_{i=1}^n
\left(R_i\otimes I_{\cE\otimes F^2(H_m)}\right)
\widetilde\Lambda_i(R_1,\ldots, R_n),
\end{equation}
for some free holomorphic functions $\Lambda_i$, $i=1,\ldots, n$,
with coefficients in $B(\cE\otimes F^2(H_m))$ and such that
$\sum_{i=1}^n \widetilde \Lambda_i(R_1,\ldots, R_n)^*\widetilde
\Lambda_i(R_1,\ldots, R_n)\leq I$.\

Now, a closer look at  relation \eqref{Gamm} reveals that (we recall
that $D_{Y^*}Y=Y D_Y$)
\begin{equation*}
\begin{split}
\widetilde \Gamma(R_1,\ldots, R_n)&= D_{\widetilde {\Psi}(0)^*}^{-1}
\left\{-D_{\widetilde {\Psi}(0)^*}\widetilde {\Psi}(0) D_{\widetilde
{\Psi}(0)}^{-1} \left[I-\widetilde {\Psi}(0)\widetilde
{\Psi}(R_1,\ldots, R_n)\right] + D_{\widetilde
{\Psi}(0)^*}^2\widetilde {\Psi}(R_1,\ldots, R_n)\right\}\\
&\qquad \qquad \times  \left[I-\widetilde {\Psi}(0)\widetilde
{\Psi}(R_1,\ldots,
R_n)\right]^{-1}D_{\widetilde {\Psi}(0)}\\
&= D_{\widetilde {\Psi}(0)^*}^{-1}\left[\widetilde
{\Psi}(R_1,\ldots, R_n)- \widetilde
{\Psi}(0)\right]\left[I-\widetilde {\Psi}(0)^* \widetilde
{\Psi}(R_1,\ldots, R_n)\right]^{-1}D_{\widetilde {\Psi}(0)}.
\end{split}
\end{equation*}
Hence, we obtain
 \begin{equation*}
\widetilde {\Psi}(R_1,\ldots, R_n)= \left[\widetilde
{\Psi}(0)+D_{\widetilde {\Psi}(0)^*}\widetilde \Gamma(R_1,\ldots,
R_n)D_{\widetilde {\Psi}(0)}^{-1}\right] \left[I+D_{\widetilde
{\Psi}(0)^*}\widetilde \Gamma(R_1,\ldots, R_n)D_{\widetilde
{\Psi}(0)}^{-1}\widetilde\Psi(0)^*\right]^{-1},
\end{equation*}
which, taking into account that  $Y^*D_{Y^*}=D_{Y} Y^*$, implies
$$
\widetilde {\Psi}(R_1,\ldots, R_n)=\widetilde{\Psi}(0)+
D_{\widetilde {\Psi}(0)^*}\widetilde \Gamma(R_1,\ldots, R_n)\left[
I+\widetilde{\Psi}(0)^* \widetilde \Gamma(R_1,\ldots,
R_n)\right]^{-1}D_{\widetilde {\Psi}(0)}.
$$
Now, using  relation \eqref{gafa} and the fact  the operators
$\widetilde{\Psi}(0)^*$ and $D_{\widetilde {\Psi}(0)^*}$ are
commuting with   $R_i\otimes I_{\cE\otimes F^2(H_m)}$, $i=1,\ldots,
n$, we obtain
\begin{equation*}
\begin{split}
\widetilde {\Psi}(R_1,\ldots, R_n)&=\widetilde{\Psi}(0)+ \left[
\sum_{i=1}^n \left(R_i\otimes I_{\cE\otimes
F^2(H_m)}\right)D_{\widetilde {\Psi}(0)^*}\widetilde
\Lambda_i(R_1,\ldots, R_n)\right]\\
&\qquad \qquad  \times \left[I+\sum_{i=1}^n \left(R_i\otimes
I_{\cE\otimes F^2(H_m)}\right) \widetilde {\Psi}(0)^*\widetilde
\Lambda_i(R_1,\ldots, R_n)\right]^{-1}D_{\widetilde {\Psi}(0)}.
\end{split}
\end{equation*}

Now, fix $W:=(W_1,\ldots, W_n)\in [B(\cH)^n]_1$. Using the
$F_n^\infty$-functional calculus for row contractions
\cite{Po-funct},  the latter relation implies
\begin{equation}
\label{Impo}
 \Psi(W)= \Psi(0)+D_{ {\Psi}(0)^*} Z\left[
I +\Psi(0)^* Z\right]^{-1}D_{\Psi(0)},
\end{equation}
where \begin{equation} \label{def-Z}
Z:=\sum_{i=1}^n (W_i\otimes
I_{\cE\otimes F^2(H_m)})\Lambda_i(W_1,\ldots, W_n).
\end{equation}

Since $ \left[\begin{matrix}
  \Psi(0) & D_{  \Psi(0)^*}\\
D_{ \Psi(0)}&  - \Psi(0)^*\end{matrix}\right] $ is a unitary
operator and
$$
\left[\begin{matrix} \Psi(W)x\\
 D_{ \Psi(0)} x - {\Psi}(0)^* Z[I+\Psi(0)^*Z]^{-1} D_{\Psi(0)}x
 \end{matrix}\right]
  =
\left[\begin{matrix}
  \Psi(0) & D_{  \Psi(0)^*}\\
D_{  \Psi(0)}&  - \Psi(0)^*\end{matrix}\right]
\left[\begin{matrix} x\\
   Z[I+\Psi(0)^*Z]^{-1} D_{\Psi(0)} x
 \end{matrix}\right],
$$
for any $x\in \cH\otimes \cE\otimes F^2(H_m)$,
  we deduce that
\begin{equation*}
\begin{split}
\|\Psi(W)x\|^2&= \|x\|^2 +\|Z[I+\Psi(0)^*Z]^{-1} D_{\Psi(0)}x\|^2-
\| D_{ \Psi(0)} x - {\Psi}(0)^*
Z[I+\Psi(0)^*Z]^{-1} D_{\Psi(0)}x\|^2\\
&= \|x\|^2 +\|Z[I+\Psi(0)^*Z]^{-1} D_{\Psi(0)}x\|^2 -\|[I+\Psi(0)^*
Z]^{-1}D_{\Psi(0)}x\|^2.
\end{split}
\end{equation*}
Hence, we have
$$
\|[I+\Psi(0)^* Z]^{-1}D_{\Psi(0)}x\|^2=\|
D_{\Psi(W)}x\|^2+\|Z[I+\Psi(0)^*Z]^{-1} D_{\Psi(0)}x\|^2,
$$
which implies
\begin{equation*}
  \|[I+\Psi(0)^* Z]^{-1}D_{\Psi(0)}x\|^2\leq
\frac{1}{1-\|Z\|^2}\| D_{\Psi(W)}x\|^2
\end{equation*}
for any $x\in \cH\otimes \cE\otimes F^2(H_m)$. Since $\Psi$ is a
strictly contractive free holomorphic function, $\|\Psi(W)\|<1$.
Consequently, $D_{\Psi(W)}$ is invertible and the latter inequality
implies
\begin{equation}
\label{ine1} \|[I+\Psi(0)^* Z]^{-1}D_{\Psi(0)}
D_{\Psi(W)}^{-1}y\|^2\leq \frac{1}{1-\|Z\|^2}\|  y\|^2
\end{equation}
for any $y\in \cH\otimes \cE\otimes F^2(H_m)$.

Since $F$ is a  strictly contractive free holomorphic function,
$F(0)$ and $F(W)$ are strict row contractions in $B(\cH\otimes
\cE)^m$. Applying Proposition \ref{connection}, we obtain
\begin{equation*}
\left\|L_{F(W),F(0)}\right\|= \sup_{e^{it}\in \TT} \left\|(I-{\bf
R}_{F(W)}  {\bf R}_{F(W)}^*)^{-1/2} (I-e^{it} {\bf
R}_{F(W)})(I-e^{it} {\bf R}_{F(0)})^{-1} (I-{\bf R}_{F(0)} {\bf
R}_{F(0)}^*)^{1/2}\right\|
\end{equation*}
where   ${\bf R}_X:=X_1^*\otimes {\bf R}_1+\cdots + X_m^*\otimes
{\bf R}_m$ is the reconstruction operator associated with the
$m$-tuple  $(X_1,\ldots. X_m)\in [B(\cH\otimes \cE)^m]_1$. Now,
using relation \eqref{defi}, one can see that $\Psi(0)={\bf
R}_{F(0)}^*$ and $\Psi(W)={\bf R}_{F(W)}^*$.  Therefore, the latter
equality becomes
\begin{equation}
\label{formula1} \left\|L_{F(W),F(0)}\right\|= \sup_{e^{it}\in
\TT}\left\| D_{\Psi(W)}^{-1} \left[I-e^{it}
\Psi(W)^*\right]\left[I-e^{it} \Psi(0)^*\right]^{-1}
D_{\Psi(0)}\right\|.
\end{equation}
Now, we need to calculate the right hand side of this equality. Note
that, using  the formula \eqref{Impo} for $\Psi(W)$, we have
\begin{equation*}
\begin{split}
D_{\Psi(W)}^{-1} &\left[I-e^{it} \Psi(W)^*\right]\left[I-e^{it}
\Psi(0)^*\right]^{-1} D_{\Psi(0)}\\
&= D_{\Psi(W)}^{-1}\left\{ I-e^{it} \Psi(0)^*-e^{it} D_{\Psi(0)}
[I+Z^*\Psi(0)]^{-1}
Z^*D_{\Psi(0)^*}\right\} [I-e^{it}\Psi(0)^*]^{-1} D_{\Psi(0)}\\
&= D_{\Psi(W)}^{-1} D_{\Psi(0)}-e^{it}D_{\Psi(W)}^{-1}
D_{\Psi(0)}[I+Z^*\Psi(0)]^{-1}
Z^*D_{\Psi(0)^*}[I-e^{it}\Psi(0)^*]^{-1} D_{\Psi(0)}\\
&= D_{\Psi(W)}^{-1} D_{\Psi(0)}[I+Z^*\Psi(0)]^{-1}
\left\{I+Z^*\Psi(0)-e^{it}Z^* D_{\Psi(0)^*} [I-e^{it}\Psi(0)^*]^{-1}
D_{\Psi(0)}\right\}\\
&=D_{\Psi(W)}^{-1} D_{\Psi(0)}[I+Z^*\Psi(0)]^{-1}
\left[I+Z^*\Theta_{\Psi(0)}(e^{it})\right],
\end{split}
\end{equation*}
where $$ \Theta_{\Psi(0)}(e^{it}):=\Psi(0)-e^{it}  D_{\Psi(0)^*}
[I-e^{it}\Psi(0)^*]^{-1} D_{\Psi(0)},\qquad e^{it}\in \TT,
$$
is the Nagy-Foia\c s characteristic function of the contraction
$\Psi(0)$ (see \cite{SzF-book}), which is contractive. Consequently,
relation \eqref{formula1} and inequality \eqref{ine1} imply
\begin{equation*}
\begin{split}
\left\|L_{F(W),F(0)}\right\|^2&=\sup_{e^{it}\in
\TT}\left\|[I+\Theta_{\Psi(0)}(e^{it})^*Z][I+\Psi(0)^*Z]^{-1}
D_{\Psi(0)} D_{\Psi(W)}^{-1}\right\|^2\\
&\leq (1+\|Z\|)^2 \frac{1}{1-\|Z\|^2}=\frac{1+\|Z\|}{1-\|Z\|}.
\end{split}
\end{equation*}
Due to  relations \eqref{def-Z} and \eqref{gafa}, and the
noncommutative von Neumann inequality (\cite{Po-von}),  we have
$\|Z\|\leq \|W\|$. Consequently, the  inequality above  implies
$$
\left\|L_{F(W),F(0)}\right\|\leq
\left(\frac{1+\|W\|}{1-\|W\|}\right)^{1/2},
$$
which proves the theorem when  $F$ is strictly contractive.

Now we consider the general case when $F$ is a contractive free
holomorphic function. Then  note that  $F_r:=rF$,  $r\in [0,1)$, is
strictly contractive and, therefore, $
\left\|L_{rF(W),rF(0)}\right\|\leq
\left(\frac{1+\|W\|}{1-\|W\|}\right)^{1/2}$, for any $r\in [0,1)$.
According to Theorem \ref{Lr},  we have $F(0)\overset{H}{{\prec}}\,
F(W)$ and
$$\|L_{F(W),F(0)}\|=\sup_{r\in [0, 1)} \|L_{rF(W),rF(0)}\|\leq
\left(\frac{1+\|W\|}{1-\|W\|}\right)^{1/2}
$$
for any $W\in [B(\cH)^n]_1$. The proof is complete.
\end{proof}

We  recall  a few facts concerning the involutive automorphisms of
the unit ball  $\BB_n$ (see  \cite{Ru}). Let $a\in \BB_n$ and
consider $\varphi_a\in
 Aut(\BB_n)$, the automorphism of the unit ball, defined by
 \begin{equation}
 \label{auto}
 \varphi_a(z):= \frac{a-Q_az-s_a (I-Q_a)z}{1-\left<z,a\right>}, \qquad z\in
 \BB_n,
 \end{equation}
where   $Q_0=0$, \, $Q_a
z:=\frac{\left<z,a\right>}{\left<a,a\right>} a$ if $ a\neq 0$, and
$s_a:=(1- \left<a,a\right>)^{1/2}$. The automorphism $\varphi_a$ has
the following properties:
\begin{enumerate}
\item[(i)]
$\varphi_a(0)=a$ and $\varphi_a(a)=0$;
\item[(ii)] $\varphi_a(\varphi_a(z))=z$, \ $z\in \BB_n$;
\item[(iii)]
$1-\|\varphi_a(z)\|_2^2=\frac{(1-\|a\|_2^2)(1-\|z\|_2^2)}{|1-\left<z,a\right>|^2}$,
\ $z\in \overline{\BB}_n$.
\end{enumerate}

\begin{theorem} \label{Sch-part}Let $F:=(F_1,\ldots, F_m)$ be
a contractive free holomorphic function  with coefficients in
$B(\cE)$.
 If $z,\xi\in \BB_n$, then $F(z)\overset{H}{\sim}\, F(\xi)$ and
$$
\|L_{F(z), F(\xi)}\|\leq \|L_{z,\xi}\|=
\left(\frac{1+\|\varphi_z(\xi)\|_2}{1-\|\varphi_z(\xi)\|_2}\right)^{1/2},$$
where $\varphi_z$ is the involutive automorphism of $\BB_n$ which
takes $0$ into $z$.  Moreover,
$$
\delta(F(z), F(\xi))\leq \delta(z,\xi)
$$
for any $z,\xi\in \BB_n$, where $\delta$ is the hyperbolic metric
defined by \eqref{hyperbolic}.
\end{theorem}

\begin{proof} For simplicity, if $z=(z_1,\ldots, z_n)\in \BB_n$ we   also  use the
notation $ z=(z_1I_\cH,\ldots, z_n I_\cH)\in [B(\cH)^n]_1$. Let
$\Psi_{ z}$ be the involutive automorphism of $[B(\cH)^n]_1$
 which takes $0$ into $z$ (see
\cite{Po-automorphism}) and let $W:=\Psi_{ z}({\bf \xi})\in
[B(\cH)^n]_1$. According to \cite{Po-automorphism}, $\Phi:=F\circ
\Psi_z$  is a contractive free holomorphic function such that
$\Phi(0)=F(z)$ and $\Phi(W)=F(\xi)$. Using  Theorem \ref{S-P}, we
deduce that
\begin{equation*}
\begin{split}
\|L_{F(\xi), F(z)}\|&=\|L_{\Phi(W),\Phi(0)}\|\leq
\left(\frac{1+\|W\|}{1-\|W\|}\right)^{1/2}\\
&=\left(\frac{1+\|\Psi_z(\xi)\|}{1-\|\Psi_z(\xi)\|}\right)^{1/2}.
\end{split}
\end{equation*}
According  to \cite{Po-automorphism}, $\Psi_z$ is a noncommutative
extension of the  involutive automorphism of $\BB_n$ that
interchanges $0$ and $z$, i.e.,   $\Psi_z(x)=\varphi_z(x)$ for $x\in
\BB_n$. Since $\|\varphi_z(\xi)\|_2=\|\varphi_\xi(z)\|_2$, we deduce
that \begin{equation} \label{max} \max\{\|L_{F(z), F(\xi)}\|,
\|L_{F(\xi), F(z)}\|\}\leq
\left(\frac{1+\|\varphi_z(\xi)\|_2}{1-\|\varphi_z(\xi)\|_2}\right)^{1/2}.
\end{equation}

Note that due to Proposition \ref{inv} and Theorem  \ref{LCC}, we
have
$$
\|L_{\xi, z}\|=\|L_{\varphi_z(\varphi_z(\xi)),
\varphi_z(0)}\|=\|L_{\varphi_z(\xi),
0}\|=\|C_{\varphi_z(\xi)}^{-1}\|.
$$
On the other hand, it was proved in \cite{Po-hyperbolic} that
$\|C_\lambda^{-1}\|=\left(\frac{1+\|\lambda\|_2}{1-\|\lambda\|_2}\right)^{1/2}$
for any $\lambda\in \BB_n$. Combining these relations, we deduce
that
$$
\|L_{\xi, z}\|=\|L_{z,\xi}\|=
\left(\frac{1+\|\varphi_z(\xi)\|_2}{1-\|\varphi_z(\xi)\|_2}\right)^{1/2}.
$$

According to \cite{Po-hyperbolic}, $\delta|_{\BB_n\times \BB_n}$
coincides with the Poincar\'e-Bergman distance on $\BB_n$, i.e.,
\begin{equation}
\label{restri}
 \delta(z,\xi)=\frac{1}{2}\ln
\frac{1+\|\varphi_z(\xi)\|_2}{1-\|\varphi_z(\xi)\|_2},\qquad
z,\xi\in \BB_n,
\end{equation}

  Using relations \eqref{max}, \eqref{restri}, and applying
Lemma \ref{OMr} and Theorem \ref{formula}, we  deduce that
  $F(z)\overset{H}{\sim}\, F(\xi)$ and $
\delta(F(z), F(\xi))\leq \delta(z,\xi) $. This completes the proof.
\end{proof}

It is well-known (see \cite{Po-poisson}, \cite{Po-holomorphic},
\cite{Po-pluriharmonic}) that if $F:=(F_1,\ldots, F_m)$ is a
contractive  ($\|F\|_\infty\leq 1$) free holomorphic function  with
coefficients in $B(\cE)$, then its boundary function is in
$R_n^\infty\bar\otimes B(\cE^{(m)},\cE)$ and, consequently, the
evaluation map $\BB_n\ni z\mapsto F(z)\in B(\cE)^{(m)}$ is a
contractive  operator-valued multiplier of the  Drury-Arveson space
(\cite{Dr}, \cite{Arv}). Moreover, any such a contractive multiplier
has  this kind of representation.

Finally, we remark that,  in the particular case when  $m=n=1$, the
second part of Theorem \ref{Sch-part} implies  Suciu's result
\cite{Su2}.

\bigskip

       %

      \end{document}